\documentclass[11pt,reqno]{amsproc} \title[The 3D inviscid limit  problem with data analytic near the boundary]{The 3D inviscid limit problem with data analytic near the boundary} \author[F.~Wang]{Fei  Wang} \address{Department of Mathematics, University of Maryland, College Park, MD 20740} \email{fwang256@umd.edu} \usepackage{fancyhdr} \usepackage{comment} \usepackage[margin=1in]{geometry} \usepackage{amsmath, amsthm, amssymb} \usepackage{times} \usepackage{graphicx, mathtools} \usepackage[usenames,dvipsnames,svgnames,table]{xcolor} \usepackage[colorlinks=true, pdfstartview=FitV, linkcolor=blue, citecolor=blue, urlcolor=blue]{hyperref} \allowdisplaybreaks \begin{document} \def\XX{X} \def\YY{Y} \def\ZZZ{Z} \def\ghoahbkseibhseir{\par} \def\intint{\int\!\!\!\!\int} \def\aa{\alpha} \def\AA{a} \def\OO{\mathcal O} \def\SS{\mathbb S} \def\CC{\mathbb C} \def\RR{\mathbb R} \def\TT{\mathbb T} \def\ZZ{\mathbb Z} \def\HH{\mathbb H^3_+} \def\RSZ{\mathcal R} \def\LL{\mathcal L} \def\SL{\LL^1} \def\ZL{\LL^\infty} \def\GG{\mathcal G} \def\eps{\varepsilon} \def\tt{\langle t\rangle} \def\erf{\mathrm{Erf}} \def\blue#1{\textcolor{blue}{#1}} \def\mgt#1{\textcolor{magenta}{#1}} \def\ff{\rho} \def\gg{G}  \def\xx{\mathcal{X}} \def\tilde{\widetilde} \def\sqrtnu{\sqrt{\nu}} \def\ww{w} \def\ft#1{#1_\xi} \def\les{\lesssim} \renewcommand*{\Re}{\ensuremath{\mathrm{{\mathbb R}e\,}}} \renewcommand*{\Im}{\ensuremath{\mathrm{{\mathbb I}m\,}}} \def\llabel#1{\notag} \newcommand{\norm}[1]{\left\|#1\right\|} \def\therere#1{}\newcommand{\nnorm}[1]{\lVert #1\rVert} \newcommand{\abs}[1]{\left|#1\right|} \newcommand{\NORM}[1]{|\!|\!| #1|\!|\!|} \newtheorem{theorem}{Theorem}[section] \newtheorem{corollary}[theorem]{Corollary} \newtheorem{proposition}[theorem]{Proposition} \newtheorem{lemma}[theorem]{Lemma} \theoremstyle{definition} \newtheorem{definition}{Definition}[section] \newtheorem{remark}[theorem]{Remark} \def\theequation{\thesection.\arabic{equation}} \numberwithin{equation}{section} \def\theequation{\thesection.\arabic{equation}} \numberwithin{equation}{section} \def\ee{\epsilon_0} \def\startnewsection#1#2{\section{#1}\label{#2}\setcounter{equation}{1}}  \def\nnewpage{\newpage} \def\sgn{\mathop{\rm sgn\,}\nolimits} \def\Tr{\mathop{\rm Tr}\nolimits}     \def\div{\mathop{\rm div}\nolimits}  \def\curl{\mathop{\rm curl}\nolimits} \def\supp{\mathop{\rm supp}\nolimits}  \def\indeq{\quad{}\quad{}} \def\period{.}                           \def\semicolon{\,;} \def\nts#1{{\cor #1\cob}} \def\colr{\color{red}} \def\colb{\color{black}} \definecolor{colorgggg}{rgb}{0.1,0.5,0.3} \definecolor{colorllll}{rgb}{0.0,0.7,0.0} \definecolor{colorhhhh}{rgb}{0,0.8,0.5} \definecolor{colorpppp}{rgb}{0.3,0.0,0.7} \def\colg{\color{colorgggg}} \def\collg{\color{colorllll}} \def\cole{} \def\colu{\color{blue}} \def\colW{\colb}    \def\comma{ {\rm ,\qquad{}} } \def\commaone{ {\rm ,\quad{}} }         \def\les{\lesssim} \def\nts#1{{\color{red}\hbox{\bf ~#1~}}} \def\blackdot{{\color{red}{\hskip-.0truecm\rule[-1mm]{4mm}{4mm}\hskip.2truecm}}\hskip-.3truecm} \def\bluedot{{\color{blue}{\hskip-.0truecm\rule[-1mm]{4mm}{4mm}\hskip.2truecm}}\hskip-.3truecm}  \def\purpledot{{\color{colorpppp}{\hskip-.0truecm\rule[-1mm]{4mm}{4mm}\hskip.2truecm}}\hskip-.3truecm} \def\greendot{{\color{colorgggg}{\hskip-.0truecm\rule[-1mm]{4mm}{4mm}\hskip.2truecm}}\hskip-.3truecm} \def\fractext#1#2{{#1}/{#2}} \def\ii{\hat\imath} \def\fei#1{\textcolor{blue}{#1}} \newcommand{\red}[1]{\color{red}{#1}\color{black}} \def\nn{\nonumber\\} \def\zz{\bar z} \begin{abstract} We consider the 3D Navier-Stokes equations in the upper half space $\HH$ with periodic boundary conditions in the horizontal directions. We prove the inviscid limit holds in the topology $L^\infty([0, T]; L^2(\HH))$ assuming the initial datum is \therere{jDi u8 rtJt TKSK jlGkGw t8 n FDx jA9 fCm iu FqMW jeox 5Akw3w Sd 8 1vK 8c4 C0O dj CHIs eHUO hyqGx3 K} analytic in the region $\{(x, y, z)\in\HH: 0\le z\le 1+\mu_0\}$ for some positive $\mu_0$   and has Sobolev regularity in the complement. \end{abstract} \maketitle \setcounter{tocdepth}{2}  \startnewsection{Introduction}{sec01} We consider the inviscid limit problem of viscous incompressible fluids on the upper half space with periodic boundary conditions in the horizontal directions. More specifically, we study the 3D Navier-Stokes equations with positive kinematic viscosity $\nu$  \begin{align}  &\partial_t u - \nu\Delta u \therere{RVg 0Pl6Z8 9X z fLh GlH IYB x9 OELo 5loZ x4wag4 cn F aCE KfA 0uz fw HMUV M9Qy eARFe3 Py 6 kQG GFx r} + u\cdot\nabla u + \nabla p = 0 \label{8ThswELzXU3X7Ebd1KdZ7v1rN3GiirRXGKWK099ovBM0FDJCvkopYNQ2a:01} \\ &\div u = 0 \label{8ThswELzXU3X7Ebd1KdZ7v1rN3GiirRXGKWK099ovBM0FDJCvkopYNQ2a:02} \\ &u|_{t=0} = u_0  \label{8ThswELzXU3X7Ebd1KdZ7v1rN3GiirRXGKWK099ovBM0FDJCvkopYNQ2a:03} \end{align} in $\HH =\TT \times \TT \times \RR_+ = \{ (x,y, z) \in \TT \times \TT \times \RR \colon z \geq 0\}$, with the no-slip boundary condition  \begin{align} &u|_{z=0} = 0. \label{8ThswELzXU3X7Ebd1KdZ7v1rN3GiirRXGKWK099ovBM0FDJCvkopYNQ2a:03b} \end{align} The 3D incompressible Euler equations, which model inviscid fluids, may be obtained by setting $\nu = 0$ in \eqref{8ThswELzXU3X7Ebd1KdZ7v1rN3GiirRXGKWK099ovBM0FDJCvkopYNQ2a:01}--\eqref{8ThswELzXU3X7Ebd1KdZ7v1rN3GiirRXGKWK099ovBM0FDJCvkopYNQ2a:03}, with the slip boundary condition given by \[u_3|_{z=0}=0.\] \ghoahbkseibhseir Whether or not the solutions of the Euler equations can describe to the leading order the solutions of the Navier-Stokes  equations  in the {\em inviscid limit} $\nu \to 0$  is a fundamental problem in mathematical fluid dynamics.  When the  fluid domain has no boundary it is well known that the solutions of  the Navier- Stokes equations converge to those of the Euler equations  in $L^\infty([0, T]; L^2(\HH))$, e.g.,~\cite{Masmoudi07, Chemin96,  ConstantinWu96}. However, in the presence of boundary  the problem  becomes very challenging. The fundamental source of difficulties lies  in the mismatch of boundary conditions between the viscous  Navier-Stokes flow (no-slip, $u_1|_{z=0} = u_2|_{z=0}=u_3|_{z=0} =  0$) and the inviscid Euler flow (slip, $u_3|_{z=0} = 0$).  Another  difficulty is due to the fast \therere{w O lDq l1Y 4NY 4I vI7X DE4c FeXdFV bC F HaJ sb4 OC0 hu Mj65 J4fa vgGo7q Y5 X tLy izY DvH TR zd9x S} growth of vorticity close to the  boundary $\partial \HH$ as $\nu \to 0$. \ghoahbkseibhseir  In~\cite{Kato84b}, Kato provided a criteria: The inviscid limit  property in the topology $L^\infty_t L^2_x$ is equivalent to the  condition \begin{align} \lim_{\nu \to 0} \int_0^T \!\!\!\! \int_{\{  z\les \nu \}} \nu |\nabla u^{NS}|^2 dx dt \to 0, \notag \end{align}   where $u^{NS}$ is a Navier-Stokes solution corresponding to the  viscosity $\nu$. For more results along these lines, we refer to~\cite{BardosTiti13, ConstantinElgindiIgnatovaVicol17,ConstantinKukavicaVicol15, Kelliher08,Kelliher17,Masmoudi98,TemamWang97b,Wang01, MaekawaMazzucato16} and references therein.  In~\cite{ConstantinVicol18,DrivasNguyen18, ConstantinBrazil18}, conditional results  have been established toward the weak $L^2_t L^2_x$ inviscid limit (against test functions compactly supported in the interior of the domain). Under certain symmetry assumptions, the vanishing viscosity limit also holds, e.g.,~\cite{BonaWu02,HanMazzucatoNiuWang12,Kelliher09,LopesMazzucatoLopes08,LopesMazzucatoLopesTaylor08,MaekawaMazzucato16,Matsui94,MazzucatoTaylor08} and references therein.   \ghoahbkseibhseir When we only make regularity assumption on the initial data without any symmetry, the problem becomes more challenging and         fewer results are available. In the seminal works~\cite{SammartinoCaflisch98a,SammartinoCaflisch98b}, Sammartino-Caflisch establish the validity of the Prandtl expansion \begin{align} u^{NS}(x,y,z, t) = u^{E}(x,y,z,t) + u^{P}\left(x,y,\frac{z}{\sqrt{\nu}},t\right) + \OO(\sqrt{\nu})  \label{8ThswELzXU3X7Ebd1KdZ7v1rN3GiirRXGKWK099ovBM0FDJCvkopYNQ2a:Prandtl} \end{align} for the solution $u^{NS}$ of \eqref{8ThswELzXU3X7Ebd1KdZ7v1rN3GiirRXGKWK099ovBM0FDJCvkopYNQ2a:01}--\eqref{8ThswELzXU3X7Ebd1KdZ7v1rN3GiirRXGKWK099ovBM0FDJCvkopYNQ2a:03b} with initial data $u^{NS}_0$ which are analytic in both the $x$ and $y$ variables on the entire half space.  Here $u^P$ is the real-analytic solution of the Prandtl boundary layer equations. As a result, the strong inviscid limit in the energy norm follows easily.  We refer the reader to~\cite{AlexandreWangXuYang14,DietertGerardVaret19,GerardVaretMasmoudi13,IgnatovaVicol16,KukavicaMasmoudiVicolWong14,KukavicaVicol13,LiuYang16,LombardoCannoneSammartino03,MasmoudiWong15,Oleinik66,SammartinoCaflisch98a} for the well-posedness theory for the Prandtl equations, \cite{GerardVaretDormy10,GuoNguyen11,GerardVaretNguyen12,LiuYang17} for the identification of ill-posed regimes, and~\cite{Grenier00,GrenierGuoNguyen15,GrenierGuoNguyen16,GrenierNguyen17,GrenierNguyen18a} for recent works which show the  invalidity of the \therere{Pf 6T ZBQR la1a 6Aeker Xg k blz nSm mhY jc z3io WYjz h33sxR JM k Dos EAA hUO Oz aQfK Z0cn 5kqYPn W7} Prandtl expansion at the level of Sobolev regularity. \ghoahbkseibhseir In a remarkable work~\cite{Maekawa14}, Maekawa proved that the inviscid limit also holds for initial datum with Sobolev regularity whose associated vorticity is supported  away from the boundary. In this paper, the author use the vorticity formulation of the Navier-Stokes system in the half space~\cite{Anderson89,Maekawa13}, taking advantage of the weak interaction between the boundary vorticity and the bulk flow inside the domain to establish the validity of the expansion~\eqref{8ThswELzXU3X7Ebd1KdZ7v1rN3GiirRXGKWK099ovBM0FDJCvkopYNQ2a:Prandtl}.  For energy based proofs of the Caflisch-Sammartino and  Maekawa results in 2D and 3D, we refer to~\cite{WangWangZhang17, FeiTaoZhang16,FeiTaoZhang18}. In~\cite{GerardVaretMaekawaMasmoudi16}, Gerard-Varet, Maekawa, and Masmoudi establish the stability in a Gevrey topology in $x$ and a Sobolev topology in $y$, of Euler+Prandtl shear flows (cf.~\eqref{8ThswELzXU3X7Ebd1KdZ7v1rN3GiirRXGKWK099ovBM0FDJCvkopYNQ2a:Prandtl}), when the Prandtl shear flow is both monotonic and concave. Also, the very recent works~\cite{GerardVaretMaekawa18,GuoIyer18,GuoIyer18a,Iyer18} establish the vanishing viscosity limit and the validity of the Prandtl expansion for the stationary  Navier-Stokes equation in certain regimes.  In 2018, Nguyen-Nguyen have found in~\cite{NguyenNguyen18} a very elegant proof of the Sammartino-Caflisch result, which for the first time completely avoids the usage of Prandtl boundary layer correctors.  In a more recent  work, I~Kukavica, V.~Vicol, and the author in~\cite{KukVicWan19} bridged the gap between the Sammartino-Caflisch~\cite{SammartinoCaflisch98a,SammartinoCaflisch98b} and the Maekawa~\cite{Maekawa14} results, by proving that the inviscid limit in the energy norm holds for initial datum $\omega_0$ which is analytic in a strip of $\OO(1)$ width close to the boundary, and is Sobolev smooth on the complement of this strip.  The main result of \cite{KukVicWan19} shows that if the initial data is analytic in a strip of constant size near the boundary, then the solution will remain analytic in a strip of constant size in an $\mathcal{O}(1)$ time interval. This result establishes the inviscid limit in the energy norm for the largest class of initial data, in the absence of structural or symmetry assumptions. \ghoahbkseibhseir \ghoahbkseibhseir \ghoahbkseibhseir \ghoahbkseibhseir The main goal of this paper is to address the 3D inviscid limit problem assuming the initial datum $\omega_0$ is analytic in a strip close to the boundary  $\partial\HH$, and is Sobolev smooth on the complement. This problem is more complicated than the 2D case considered in~\cite{KukVicWan19} because of the nonlinear vortex stretching term $\omega\cdot\nabla u$ and the incompatibility of the third component of the the vorticity with the other components. Compared with the 2D case, the main difference in 3D is that the vorticity \therere{ 1 vCT 69a EC9 LD EQ5S BK4J fVFLAo Qp N dzZ HAl JaL Mn vRqH 7pBB qOr7fv oa e BSA 8TE btx y3 jwK3 v2} is a vector instead of a scalar function, as a result of which we have the extra vortex stretching term $\omega\cdot\nabla u$ in the nonlinearity. As it turns out the terms $\omega\cdot\nabla u$ and $u\cdot\nabla \omega$ have very different nature since $\omega\partial_{z}u$ could potentially be of the same size os $\omega\cdot\omega$ which we are expecting to be of the order of $\mathcal{O}\left(1/\nu\right)$ near the boundary. Since our weight function is designed to only capture the $\mathcal{O}\left(1/\sqrt{\nu}\right)$ growth of the vorticity near the boundary, the quadratic nonlinearity could be out of control. In order to overcome this difficulty, we need to give a more precise  description of the vorticity near the boundary. From the equation for the vorticity (see~\eqref{eq:vor}--\eqref{bdry:ver}) we can see there are mixed boundary conditions, i.e., the Neumann condition for the horizontal components and the Dirichlet boundary condition for the vertical component, as a result of which we may expect better behavior of the vertical vorticity. In fact we have $\omega_h \sim \mathcal{O}(1/\sqrt{\nu})$ and $\omega_z\sim \mathcal{O}(1)$ near the boundary. To capture this difference, we design the norm $\bar{\mathcal{X}}$ for the horizontal components and $\mathcal{X}$ for the vertical component (see~\eqref{nor:X:mu}). Another difficulty is that the vortical component and the horizontal \therere{44 dlfwRL Dc g X14 vTp Wd8 zy YWjw eQmF yD5y5l DN l ZbA Jac cld kx Yn3V QYIV v6fwmH z1 9 w3y D4Y ez} components of vorticity are coupled together and we have the nonlinearity $\omega\cdot u_3$, which could potentially be of the size $\mathcal{O}\left(1/\sqrt{\nu}\right)$, in the equation for $\omega_3$. However, there is not weight function $w(x)$ in the $\mathcal{X}$-norm to compensate this growth.  Roughly speaking, we  need to explore the structure of the equation and trade this growth with an extra horizontal regularity which we could afford. \ghoahbkseibhseir \ghoahbkseibhseir \startnewsection{Preliminaries}{sec02} \subsection{Notation} In this paper, the differential operators $\partial_{1}$, $\partial_{2}$, and $\partial_{3}$ stand for the derivatives in $x$, $y$, and $z$ directions respectively. We use the notation $(x_1, x_2, x_3)=(x_h, x_3)=(x_h, x_z)\in\RR^3$, where $x_h=(x_1, x_2)$ (h for horizontal). Correspondingly we write $\nabla=(\nabla_h, \partial_{z})=(\nabla_h, \partial_{3})$. For an index $\alpha\in\mathbb{Z}^3$, we denote $\nabla^\alpha=\partial_{x}^{\alpha_1}\partial_{y}^{\alpha_2}\partial_{z}^{\alpha_3}$. The square root of the Laplace operator in the  horizontal directions is denoted by $\Lambda_h = \sqrt{\Delta_h}$. Note that by definition we have $\Lambda_h = |\nabla_h|$. We define the conormal differential operator $D=(\partial_{x}, \partial_{y}, z\partial_{z})$ and $D^\alpha=\partial_{x}^{\alpha_1}\partial_{y}^{\alpha_2}(z\partial_{z})^{\alpha_3}$  for an index $\alpha\in\mathbb{Z}^3$ (not to be confused with $\nabla^\alpha$). We use $f_\xi(z) \in \CC$ to denote the Fourier transform of $f$ with respect to the horizontal variables $x$ and $y$ at frequency $\xi=(\xi_1, \xi_2) \in \ZZ^2$ and $u_{i,\xi}$ for the Fourier transform of $u_i$ in $x$ and $y$ for $i=1, 2, 3$.  For $\mu >0$  we define the complex domain  \[\Omega_\mu=\{z\in \CC: 0 \leq \Re z \leq 1, |\Im z|\le \mu \Re z\} \cup \{z\in \CC: 1 \leq \Re z \leq 1+\mu, |\Im z|\le 1 + \mu -  \Re z  \}.\] For $z\in \Omega_\mu$ we represent exponential terms of the form $e^{\ee(1+\mu-\Re z)_+|\xi|}$ simply as $e^{\ee(1+\mu-z)_+|\xi|}$. That is, in order to simplify the notation we write $y$ instead of $\Re y$ inside the exponential. We consider $\nu\in(0,1]$ as a small parameter and we assume $t \in (0,1]$ throughout. The implicit constants in $\les$ depend only on $\mu_0$ and $\theta_0$ (cf.~\eqref{8ThswELzXU3X7Ebd1KdZ7v1rN3GiirRXGKWK099ovBM0FDJCvkopYNQ2a101}), and are thus  universal.  \ghoahbkseibhseir  \ghoahbkseibhseir \subsection{Vorticity formulation deduction} Compared with the 2D case, the vorticity formulation of the 3D~Navier-Stokes equations we use in this paper is more involved, since we have the Neumann boundary condition~\eqref{bdry:hor} for the horizontal component $\omega_h$ and the Dirichlet boundary condition~\eqref{bdry:ver} for the vertical component $\omega_3$. Also the vorticity equation contains the vortex stretching term in~\eqref{eq:vor} which is absent in the 2D case. Moreover, the Biot-Savart law is  the curl of the Newtonian potential of $\omega$, resulting in more complexity.  \ghoahbkseibhseir Taking the curl of equation~\eqref{8ThswELzXU3X7Ebd1KdZ7v1rN3GiirRXGKWK099ovBM0FDJCvkopYNQ2a:01} gives   \begin{equation} \notag   \omega_t + u\cdot\nabla\omega -\nu\Delta\omega =\omega\cdot\nabla u,   \end{equation} where $u$ is recovered by the Biot-Savart law    \begin{equation}   \label{bio:sav}   u=\curl W   \end{equation}  with $W=(-\Delta)^{-1}\omega$. For the   boundary conditions in this setting note that   \begin{equation}    \notag   \omega_3|_{z=0}= \partial_{1}u_2-\partial_{2}u_1|_{z=0}=0.   \end{equation} For the horizontal components, we have   \begin{align} \notag   u_1|_{z=0}=\partial_{2}W_3-\partial_{3}W_2|_{z=0}=-\partial_{3}W_2|_{z=0},   \end{align}   from where we obtain   \begin{align} \notag   -\partial_{t}u_1|_{z=0}=\partial_{3}(-\Delta)^{-1}\partial_{t}\omega_2|_{z=0}   =\partial_{3}(-\Delta)^{-1}(-u\cdot\nabla\omega_2+\omega\cdot\nabla   u+\nu\Delta\omega)|_{z=0}.   \end{align}  Since   \begin{equation}   \notag   \partial_{3}(-\Delta)^{-1}\Delta\omega_2|_{z=0}   =   -(\partial_{3}\omega_2+\Lambda_h\omega_2)|_{z=0}   \end{equation}   where $\Lambda_h=\sqrt{-\Delta_h}$,   by   $u_1|_{z=0}=0$ we further get   \therere{R M9 BduE L7D9 2wTHHc Do g ZxZ WRW Jxi pv fz48 ZVB7 FZtgK0 Y1 w oCo hLA i70 NO Ta06 u2sY GlmspV l2 } \begin{align} \notag   \nu(\partial_{3}+\Lambda_h)\omega_2|_{z=0}   =\partial_{3}(-\Delta)^{-1}(-u\cdot\nabla\omega_2+\omega\cdot\nabla   u_2)|_{z=0}.   \end{align}   A similar computation for $\omega_1$   gives    \begin{align} \notag   \nu(\partial_{3}+\Lambda_h)\omega_1|_{z=0}   =\partial_{3}(-\Delta)^{-1}(-u\cdot\nabla\omega_1+\omega\cdot\nabla    u_1)|_{z=0}.   \end{align}     In summary, we have the vorticity   formulation for the $3$D N--S equations with mixed boundary   conditions   \begin{align}   &\omega_t + u\cdot\nabla\omega   -\nu\Delta\omega =\omega\cdot\nabla u  \label{eq:vor}  \\  &   \nu(\partial_{3}+\Lambda_h)\omega_h|_{z=0}   =\partial_{3}(-\Delta)^{-1}(-u\cdot\nabla\omega_h+\omega\cdot\nabla   u_h)|_{z=0}  \label{bdry:hor}  \\ &   \omega_3|_{z=0}=0.      \label{bdry:ver}   \end{align} \ghoahbkseibhseir \ghoahbkseibhseir \subsection{Integral representation of \therere{x y0X B37 x43 k5 kaoZ deyE sDglRF Xi 9 6b6 w9B dId Ko gSUM NLLb CRzeQL UZ m i9O 2qv VzD hz v1r6 spS} the solution to the Navier-Stokes equations} In this section, we derive the mild formulation of the system~\eqref{eq:vor}--\eqref{bdry:ver}.  Due to the combined boundary condition~\eqref{bdry:hor}--\eqref{bdry:ver} in 3D, the Green's kernel is given as a matrix in~\eqref{gre:fun} instead of a scalar function in 2D case.  For $\xi \in \ZZ^2$, let \begin{equation} \notag   N_\xi(s, z)= (\omega\cdot\nabla u -u\cdot\nabla\omega)_\xi(s, z)   \end{equation}   stand for the Fourier transform in the $x$ and  $y$ variables of the  nonlinear terms in the vorticity formulation of the Navier-Stokes  system. We denote by   \begin{equation} \notag  B_{h,\xi}(s)=(\partial_{z}(-\Delta)^{-1}(\omega\cdot\nabla u_h  -u\cdot\nabla\omega_h))_\xi(s)|_{z=0}   \end{equation}   the boundary  condition for the horizontal components. Then the  system~\eqref{eq:vor}--\eqref{bdry:ver} is rewritten as a Stokes  problem   \begin{subequations}   \label{eq:Stokes}   \begin{align}  \partial_{t}\omega_\xi - \nu\Delta_\xi\omega_\xi &= N_\xi \nonumber\\   \nu(\partial_{z}+|\xi|)\omega_{h,\xi}|_{z=0} &= B_{h,\xi}    \label{Neumann}  \\   \omega_{3,\xi}|_{z=0} & = 0.   \label{Dirichlet}   \end{align}   \end{subequations} where   $\Delta_\xi= -|\xi|^2+\partial_{z}^2$  is considered with a   Dirichlet boundary condition at $z=0$.  To obtain estimates in the   analytic region, we use the mild formulation for the above   system. Denoting the Green's function for this Stokes   system~\eqref{eq:Stokes} by $G_\xi(t, z, \zz)$, then Green's    function is of the form   \begin{equation}   \label{gre:fun}   G_\xi   =   \begin{pmatrix}   G_{1,\xi} & 0 & 0\\   0 & G_{1,\xi} &   0\\   0 & 0 &G_{2,\xi}   \end{pmatrix}   \end{equation} where the   kernels $G_{1,\xi}$ and $G_{2,\xi}$ are the Green's functions for a   heat equation with the Neumann~\eqref{Neumann} and the   Dirichlet~\eqref{Dirichlet} boundary conditions respectively in the   half space. The kernel $G_{1, \xi}$ is given in   Theorem~\ref{Green:fun} below and $G_{2, \xi}$ has the explicit   formula   \begin{equation} \notag   G_{2,\xi}(t, z, \zz) =  \frac{1}{\sqrt{\nu t}}\left(e^{-\frac{(y-z)^2}{4\nu t}} - e^{-\frac{(y+z)^2}{4\nu t}} \right)e^{-\nu|\xi|^2t}.   \end{equation}  Using the Green's kernel $G_\xi$, we may represent   the solution for the Stokes system~\eqref{eq:Stokes}  as   \begin{align}   \label{kernel:est}   \omega_\xi (t,z) = &   \int_0^\infty G_\xi(t, z, \zz)\omega_{0\xi}(\zz)\,d\zz + \int_0^t   \int_0^\infty G_\xi(t-s, z, \zz)  N_\xi(s, \zz) \,d\zz ds   \nonumber\\&   + \int_0^t G_\xi(t-s, z,0) (B_{h,\xi}(s), 0) \,ds   \end{align} where $\omega_{0\xi}(z)$ is the Fourier transform of the    initial datum. \begin{remark} \label{R03} In the Duhamel formula   \eqref{kernel:est}, we should understand each term on the right side   as an integral in the complex region $z,\zz\in \Omega_{\mu}\cup   [1+\mu,\infty)$ as explained in \cite{NguyenNguyen18}, i.e., the   path is inside the complex domain from $0$ to $\infty$.  For $z\in \Omega_\mu$, we may find $\theta \in [0,\mu)$ such that $z \in \partial\Omega_\theta$. If $\Im z \geq 0$, we integrate over the complex contour $\gamma_\theta^+ = (\partial \Omega_\theta \cap \{ \zz \colon \Im \zz \geq 0 \})  \cup [1+\theta ,\infty)$. Otherwise we choose the path which is symmetric with  $\gamma_\theta^+$ to the line $\{ \zz \colon \Im \zz = 0 \}$.  Moreover, the Green's function $G_\xi(t,z,\zz)$ from Lemma~\ref{Green:fun}, which appears in \eqref{kernel:est}, has a natural extension to the complex domain $\Omega_\mu \cup [1+\mu,\infty)$, by complexifying the heat kernels involved. Since for $z\in \Omega_\mu$ we have $|\Im z| \leq \mu \Re z$, for $\mu$ small, we have that $|z|$ is comparable to $\Re z$. Therefore, the upper bounds we have available  for the complexified heat kernel $\tilde H_\xi$ and for the residual kernel $R_\xi$ may be written in terms of $\Re z, \Re \zz \geq 0$. Because of this, we perform all the estimates in the real variables and the corresponding estimates for the complex variables follow similarly. \end{remark}  \ghoahbkseibhseir An estimate of the Green's function $G_{1,\xi}$ of the Stokes system is given in the next lemma. For its proof, we refer to~\cite[Proposition 3.3 and Section 3.3]{NguyenNguyen18}. \ghoahbkseibhseir \cole \begin{lemma}[Nguyen \& Nguyen] \label{Green:fun} We have the following representation of the Green's function $G_{1, \xi}$    \begin{equation}   \llabel{8Th sw ELzX U3X7 Ebd1Kd Z7 v 1rN 3Gi irR XG KWK0 99ov BM0FDJ Cv k opY NQ2 8ThswELzXU3X7Ebd1KdZ7v1rN3GiirRXGKWK099ovBM0FDJCvkopYNQ2a30}   G_{1, \xi} = \tilde H_\xi + \tilde R_\xi   \end{equation} where   \begin{align}   \llabel{aN9 4Z 7k0U nUKa   mE3OjU 8D F YFF okb SI2 J9 V9gV lM8A LWThDP nP u 3EL 7HP   8ThswELzXU3X7Ebd1KdZ7v1rN3GiirRXGKWK099ovBM0FDJCvkopYNQ2a31}    \tilde H_\xi(t, z ,\zz) = \frac{1}{\sqrt{\nu   t}}\left(e^{-\frac{(z-\zz)^2}{4\nu t}} + e^{-\frac{(z+\zz)^2}{4\nu   t}} \right)e^{-\nu|\xi|^2t}   \end{align} is the one dimensional   heat kernel for the half space with homogeneous Neumann boundary   condition; the residual kernel may be further decomposed as  $\tilde   R_{\xi}= R^{(1)}_{\xi}+ R^{(2)}_{\xi}$, with the two kernels satisfying the bounds   \begin{equation}       |\partial_{z}^{k}R^{(1)}_\xi(t, z, \zz)|   \lesssim b^{k+1} e^{-\theta_0b(z+\zz)}     \comma k\in{\mathbb N}_0   \label{8ThswELzXU3X7Ebd1KdZ7v1rN3GiirRXGKWK099ovBM0FDJCvkopYNQ2a101}   \end{equation} and   \begin{equation}       |\partial_{z}^{k}   R^{(2)}_\xi(t, z, \zz)|   \lesssim \frac{1}{(\nu t)^{(k+1)/2}}   e^{-\frac{(z+\zz)^2}{\nu t}}e^{-\frac{\nu|\xi|^2t}{8}}    \comma k\in{\mathbb N}_0    \label{8ThswELzXU3X7Ebd1KdZ7v1rN3GiirRXGKWK099ovBM0FDJCvkopYNQ2a326}    \end{equation} where  $\theta_0>0$ is a constant and the boundary    remainder coefficient  is given by $b=b(\xi, \nu)=    |\xi|+\sqrt{\nu}^{-1}$. The implicit constants in     \eqref{8ThswELzXU3X7Ebd1KdZ7v1rN3GiirRXGKWK099ovBM0FDJCvkopYNQ2a101}    and    \eqref{8ThswELzXU3X7Ebd1KdZ7v1rN3GiirRXGKWK099ovBM0FDJCvkopYNQ2a326}    depend on $k$. \end{lemma} \colb \ghoahbkseibhseir \begin{remark} \label{R01} From the estimates \eqref{8ThswELzXU3X7Ebd1KdZ7v1rN3GiirRXGKWK099ovBM0FDJCvkopYNQ2a101} and \eqref{8ThswELzXU3X7Ebd1KdZ7v1rN3GiirRXGKWK099ovBM0FDJCvkopYNQ2a326}, the residual kernel $R_{\xi}$ satisfies     \begin{align}    | (z\partial_{z})^k   R_\xi (t,z,\zz)|     &\lesssim b \, ((z b)^k +1) e^{-\theta_0b(z+\zz)}    +   \biggl(\left(\frac{z}{\sqrt{\nu t}}\right)^k +1 \biggr)           \frac{1}{\sqrt{\nu t}}e^{- \theta_0 \frac{(z+\zz)^2}{\nu t}}e^{-\frac{\nu|\xi|^2t}{8}}  \notag\\ &\lesssim be^{-\frac{\theta_0}{2}b(z+\zz)}     +  \frac{1}{\sqrt{\nu t}}e^{- \frac{\theta_0}{2} \frac{(z+\zz)^2}{\nu t}}e^{-\frac{\nu|\xi|^2t}{8}}    \llabel{D2V Da ZTgg zcCC mbvc70 qq P  cC9 mt6 0og cr TiA3 HEjw TK8ymK eu J Mc4 q6d 8ThswELzXU3X7Ebd1KdZ7v1rN3GiirRXGKWK099ovBM0FDJCvkopYNQ2a33} \end{align}   for $k \in \{0,1,2\}$.  \end{remark} \ghoahbkseibhseir \ghoahbkseibhseir \startnewsection{Main Results}{sec03} \subsection{Norms} In this paper, we use two types of analytic norms: the $L^\infty$ based $X$ norm and the $L^1$ based $\YY$  norm defined in \eqref{8ThswELzXU3X7Ebd1KdZ7v1rN3GiirRXGKWK099ovBM0FDJCvkopYNQ2a17} and \eqref{8ThswELzXU3X7Ebd1KdZ7v1rN3GiirRXGKWK099ovBM0FDJCvkopYNQ2a25} respectively.  \therere{l jwNhG6 s6 i SdX hob hbp 2u sEdl 95LP AtrBBi bP C wSh pFC CUa yz xYS5 78ro f3UwDP sC I pES HB1 qFP} Compared with the uniform $X$ analytic norm in 2D introduced in~\cite{KukVicWan19}, we need to define an $\bar\xx$ analytic norm for the horizontal  components and an $\xx$ norm for the vertical component due to the very different behaviors (i.e., the vorticity in vertical direction is much smaller than in the horizontal plane). These norms are designed to capture the analytic features of a  solution in the domain $\Omega_\mu$ with the radius of analyticity in horizontal directions decreasing linearly. For the $\YY$ analytic norm, we need one more derivative in the horizontal direction than the 2D case, which is the key to compensate the incompatibility of the two norms $\bar \xx$  and $\xx$. We also use a weighted $H^5$ norm to describe the Sobolev regularity at an $\OO(1)$ distance away from the boundary. \ghoahbkseibhseir In order to define the weighted $L^\infty$ based analytic norm $X$, we introduce a $L^\infty$ norm   \begin{equation}   \lVert f\rVert_{\ZL_{\mu,\nu}} = \sup_{z\in\Omega_\mu} \ww(\Re z)   |f(z)|   \llabel{Vz2 00 XnYU tLR9 GYjPXv FO V r6W 1zU K1W bP ToaW   JJuK nxBLnd 0f t DEb Mmj   8ThswELzXU3X7Ebd1KdZ7v1rN3GiirRXGKWK099ovBM0FDJCvkopYNQ2a15}   \end{equation} over the domain $\Omega_\mu$ for a complex function   $f$, where the  weight function $w \colon [0,1+\mu_0] \to [0,1]$  is   given by    $ w(z)     =\max( \sqrtnu, z)$ for $0\le z\leq 1$ and   $1$ otherwise; cf.~Proposition~\ref{wei} below for the needed    properties of the weight function. Note that the $\mu$ independence   of the norm $\lVert f\rVert_{\ZL_{\mu,\nu}}$ is encoded in $w$. In order to simplify the notation, we suppress the $\Re$ symbol when there is no confusion. For instance, we write $w(z)$ instead of $w(\Re z)$ in the above formula. We also use the $L^\infty$ based norm without weight    \begin{equation}   \lVert f\rVert_{\ZL_{\mu}} =   \sup_{z\in\Omega_\mu} |f(z)|   \llabel{4lo HY yhZy MjM9 1zQS4p 7z 8   eKa 9h0 Jrb ac ekci rexG 0z4n3x z0 Q OWS vFj   8ThswELzXU3X7Ebd1KdZ7v1rN3GiirRXGKWK099ovBM0FDJCvkopYNQ2a15}   \end{equation} for the vertical component of a function $f$. For a   sufficiently small constant  $\ee \in (0,1)$ to be determined below,   using the $\ZL_{\mu,\nu}$ and $\ZL_{\mu}$ norms, we define   \begin{equation}   \llabel{3jL hW XUIU 21iI AwJtI3 Rb W a90 I7r zAI   qI 3UEl UJG7 tLtUXz w4 K QNE TvX   8ThswELzXU3X7Ebd1KdZ7v1rN3GiirRXGKWK099ovBM0FDJCvkopYNQ2a18}   \lVert f \rVert_{\bar\xx_{\mu}} = \sum_{\xi \in \ZZ^2} \lVert   e^{\ee(1+\mu-z)_+|\xi|}f_{\xi}\rVert_{\ZL_{\mu,\nu}}  \,    \end{equation}  and    \begin{equation}   \llabel{zqW au jEMe nYlN IzLGxg B3 A uJ8  6VS 6Rc PJ 8OXW w8im tcKZEz Ho p 84G 1gS   8ThswELzXU3X7Ebd1KdZ7v1rN3GiirRXGKWK099ovBM0FDJCvkopYNQ2a:X}   \lVert  f \rVert_{\xx_{\mu}} = \sum_{\xi \in \ZZ^2} \lVert  e^{\ee(1+\mu-z)_+|\xi|}f_{\xi}\rVert_{\ZL_{\mu}}.  \,  \end{equation}  Note that the constant $\ee$ depends only on the  parameter $\theta_0$. We need to differentiate the horizontal  components of $f$ from the vertical one and further denote  \begin{equation}   \label{nor:X:mu}   \lVert f \rVert_{X_{\mu}} =  \lVert f_h \rVert_{\bar\xx_{\mu}} + \lVert f_3 \rVert_{\xx_{\mu}}.   \end{equation} With $t$ as in \eqref{8ThswELzXU3X7Ebd1KdZ7v1rN3GiirRXGKWK099ovBM0FDJCvkopYNQ2a:time:restrict}, we define the analytic $X$ norm as   \begin{align}   \lVert f\rVert_{X(t)}   = \sup_{\mu<\mu_0-\gamma t}    \biggl( \sum_{0\le |\alpha| \le1}\lVert D^\alpha f\rVert_{X_\mu} +    \sum_{|\alpha|=2} (\mu_0-\mu-\gamma t)^{1/2+\AA}\lVert D^\alpha f\rVert_{X_\mu} \biggr)    \, \label{8ThswELzXU3X7Ebd1KdZ7v1rN3GiirRXGKWK099ovBM0FDJCvkopYNQ2a17}  \end{align} where $\aa\in\left(0,     \frac{1}{2} \right) $ is a fixed   constant and $\gamma>0$ is a sufficiently large constant to be determined. It depends only on $\mu_0$ and the size of the initial datum. Throughout the paper we let $t$ obey \begin{align}  t \in \left(0, \frac{\mu_0}{2\gamma} \right)\,. \label{8ThswELzXU3X7Ebd1KdZ7v1rN3GiirRXGKWK099ovBM0FDJCvkopYNQ2a:time:restrict} \end{align}  \ghoahbkseibhseir For the definition of the analytic $L^1$ based norm, the weight function is not needed, and thus we define   \begin{equation}   \llabel{As0 PC owMI 2fLK TdD60y nH g 7lk NFj JLq Oo Qvfk fZBN G3o1Dg Cn 9 hyU h5V 8ThswELzXU3X7Ebd1KdZ7v1rN3GiirRXGKWK099ovBM0FDJCvkopYNQ2a24}   \lVert f\rVert_{\SL_\mu} = \sup_{0\le\theta<\mu} \lVert   f\rVert_{L^1(\partial\Omega_\theta)}   \,   \end{equation} where $f$   is a complex valued function over $\Omega_\mu.$ Using the $\SL_\mu$   norm we introduce \begin{equation}   \llabel{SP5 z6 1qvQ wceU dVJJsB   vX D G4E LHQ HIa PT bMTr sLsm tXGyOB 7p 2 Os4 3US   8ThswELzXU3X7Ebd1KdZ7v1rN3GiirRXGKWK099ovBM0FDJCvkopYNQ2a27}   \lVert f \rVert_{\YY_\mu} = \sum_\xi \lVert    e^{\ee(1+\mu-z)_+|\xi|}f_\xi\rVert_{\SL_\mu}.   \end{equation}  Now   we  are ready to define the analytic $Y$ norm as   \begin{align}   \label{8ThswELzXU3X7Ebd1KdZ7v1rN3GiirRXGKWK099ovBM0FDJCvkopYNQ2a25}   \lVert f\rVert_{\YY(t)}=  \sup_{\mu<\mu_0-\gamma t}   &\biggl(\sum_{0\le |\alpha|\le1} \norm{D^\alpha (1+ |\nabla_h|)   f}_{Y_\mu}  +\sum_{|\alpha|=2}  (\mu_0-\mu-\gamma t)^\AA   \norm{D^\alpha (1+ |\nabla_h|)f}_{Y_\mu}   \biggr)    \,.   \end{align} \ghoahbkseibhseir   \ghoahbkseibhseir Next we introduce   two kinds of Sobolev norms. The first type is  weighted $L^2$ in   $y$, $\ell^1$ in $\xi$ norm  $S_\mu$ given by   \begin{equation}   \llabel{bq5 ik 4Lin 769O TkUxmp I8 u GYn fBK bYI 9A QzCF w3h0 geJftZ   ZK U 74r Yle   8ThswELzXU3X7Ebd1KdZ7v1rN3GiirRXGKWK099ovBM0FDJCvkopYNQ2a20}   \lVert f \rVert_{S_\mu} = \sum_\xi \lVert z f_\xi\rVert_{L^2(z\ge   1+\mu)}   \,,   \end{equation} which is compatible with the above   analytic norms. We also need a second type of Sobolev norm, i.e., an   $L^2$ based $\ZZZ$-norm inside the half space as  \begin{align}    \norm{f}_{\ZZZ}     =  \sum_{0\le |\alpha|\le5} \norm{\nabla^{\aa}   f}_S \,    \llabel{ajm km ZJdi TGHO OaSt1N nl B 7Y7 h0y oWJ ry rVrT   zHO8 2S7oub QA W x9d z2X   8ThswELzXU3X7Ebd1KdZ7v1rN3GiirRXGKWK099ovBM0FDJCvkopYNQ2a54}   \end{align} where the $S$-norm is a weighted $L^2$ norm (with respect to both $x$ and $y$):  \begin{equation}   \llabel{YWB e5 Kf3A LsUF vqgtM2 O2 I dim rjZ 7RN 28 4KGY trVa WW4nTZ XV b RVo Q77 8ThswELzXU3X7Ebd1KdZ7v1rN3GiirRXGKWK099ovBM0FDJCvkopYNQ2a19}   \lVert f\rVert_{S}^2 = \norm{y f}_{L^2(y\geq 1/2)}^2 =  \sum_\xi \lVert y f_\xi\rVert_{L^2(y\ge 1/2)}^2.   \end{equation} Note that we should think the $\ZZZ$-norm as  a weighted version of the Sobolev $H^5$ norm. Lastly, for fixed $\mu_0,\gamma>0$, and with $t$ which obeys \eqref{8ThswELzXU3X7Ebd1KdZ7v1rN3GiirRXGKWK099ovBM0FDJCvkopYNQ2a:time:restrict}, we introduce the notation \begin{align} \NORM{f}_t = \norm{f}_{X(t)} + \norm{f}_{\YY(t)} + \norm{f}_{\ZZZ} \notag \end{align} for the {\em cumulative} time-dependent norm used in this paper.    \ghoahbkseibhseir  \ghoahbkseibhseir \subsection{Main results} Let  $\omega = \curl u$ be the vorticity associated to the velocity field $u$. The following are the main results of the paper. \ghoahbkseibhseir \cole \begin{theorem} \label{T01} Let $\mu_0>0$ and \therere{ SW 5tt0 I7oz jXun6c z4 c QLB J4M NmI 6F 08S2 Il8C 0JQYiU lI 1 YkK oiu bVt fG uOeg Sllv b4HGn3 bS Z} assume that $\omega_0$ is such that   $\NORM{\omega_0}_0 \le  M <\infty$.  Then there exists a $\gamma>0$ and a time $T>0$ depending on  $M$ and $  \mu_0$,  such that the solution $\omega$ to the system \eqref{eq:vor}--\eqref{bdry:ver} satisfies   \begin{equation} \llabel{hVL X6 K2kq FWFm aZnsF9 Ch p 8Kx rsc SGP iS tVXB J3xZ cD5IP4 Fu 9 Lcd TR2 8ThswELzXU3X7Ebd1KdZ7v1rN3GiirRXGKWK099ovBM0FDJCvkopYNQ2a35} \sup_{t\in [0,T]} \NORM{\omega(t)}_t \le C M   \end{equation} where the constant $C$ only depends $\mu_0$ and $\theta_0$ (cf.~\eqref{8ThswELzXU3X7Ebd1KdZ7v1rN3GiirRXGKWK099ovBM0FDJCvkopYNQ2a101}). \end{theorem} \colb \ghoahbkseibhseir From the above result we further obtain the inviscid limit. \cole \begin{theorem} \label{T02} Let $\omega_0$ be as in  Theorem~\ref{T01}. Denote by $u^\nu$ the solution of the Navier-Stokes equation \eqref{8ThswELzXU3X7Ebd1KdZ7v1rN3GiirRXGKWK099ovBM0FDJCvkopYNQ2a:01}--\eqref{8ThswELzXU3X7Ebd1KdZ7v1rN3GiirRXGKWK099ovBM0FDJCvkopYNQ2a:03b} with viscosity $\nu>0$, defined on $[0,T]$, where $T$ is as given in Theorem~\ref{T01}. Also, denote by $\bar u$ the solution of the Euler equations with initial datum $\omega_0$. Then we have \begin{align} \lim_{\nu \to 0} \sup_{t \in [0,T]} \norm{u^\nu(t) - \bar u(t)}_{L^2(\HH)} = 0  \,.   \notag \end{align} \end{theorem} \colb \ghoahbkseibhseir \ghoahbkseibhseir \section{Estimates for the nonlinearity}  \label{sec-N} In this section, we estimate $ D^{\aa}N(s)$ with $0\le |\aa|\le 1$ in the $X_\mu$ norm and  $ D^{\aa}(1+ |\nabla_h|)N(s)$ with $0\le |\aa|\le 1$ in the $\YY_\mu$ and $S_\mu$ norms, which are among the main differences with the 2D case.  A slight modification of ~\cite[Lemma~2.4]{Maekawa14} implies the following representation formulas of the velocity field in terms of the  vorticity. \ghoahbkseibhseir \cole \begin{lemma} \label{der:str} The following representations hold   \begin{align}   \begin{split}   (\partial_{i}(-\Delta)^{-1}\omega)_{\xi}(s, z)   =&    \frac{\ii}{2}\frac{\xi_i}{|\xi|}\Big(\int_0^z   e^{-|\xi|(z-\zz)}(1-e^{-2|\xi| \zz})\omega_\xi(s, \zz) \,d\zz   \\&\indeq +  \int_z^\infty e^{-|\xi|(\zz-z)}(1-e^{-2|\xi|   z})\omega_\xi(s, \zz) \,d\zz\Big)   \end{split}   \label{8ThswELzXU3X7Ebd1KdZ7v1rN3GiirRXGKWK099ovBM0FDJCvkopYNQ2a124}   \end{align} for $i=1, 2$ and   \begin{align}   \begin{split}   (\partial_{3}(-\Delta)^{-1}\omega)_{\xi}(z) =& \frac{1}{2}\Big(-\int_0^z e^{-|\xi|(z-\zz)}(1-e^{-2|\xi| \zz})\omega_\xi(s, \zz) \,d\zz    \\&\indeq +    \int_z^\infty e^{-|\xi|(\zz-z)}(1+e^{-2|\xi|   z})\omega_\xi(s, \zz) \,d\zz\Big)   \end{split}   \label{8ThswELzXU3X7Ebd1KdZ7v1rN3GiirRXGKWK099ovBM0FDJCvkopYNQ2a123}   \end{align} where $\ii$ is the imaginary unit. \end{lemma} \colb   \ghoahbkseibhseir As in Remark~\ref{R03} above, the Biot-Savart law   of Lemma~\ref{bio:sav} also holds for $z$ in the complex domain   $\Omega_\mu \cup [1+\mu,\infty)$. If $z \in \partial \Omega_\theta$   for some $\theta \in [0,\mu)$, and say $\Im z \geq 0$, then the    integration from $0$ to $z$ in   \eqref{8ThswELzXU3X7Ebd1KdZ7v1rN3GiirRXGKWK099ovBM0FDJCvkopYNQ2a123}--\eqref{8ThswELzXU3X7Ebd1KdZ7v1rN3GiirRXGKWK099ovBM0FDJCvkopYNQ2a124}   is an integration over the complex line $\partial \Omega_\theta \cap   \{\bar z\colon \Im \bar z \geq 0, \Re \bar z \leq \Re z \}$, while   the integration from $z$ to $\infty$ is an integration over $(   \partial \Omega_\theta  \cap \{\bar z\colon \Im\bar z \geq 0,   \Re\bar z \leq \Re z \leq 1+\theta \}) \cup [1+\theta, \infty)$.  \ghoahbkseibhseir The main estimate concerning the $X_\mu$ norm is the following.  \cole \begin{lemma} \label{L01} Let $\mu \in (0, \mu_0 - \gamma s)$ be arbitrary.  We have the inequalities   \begin{align}   \lVert N(s)\rVert_{X_{\mu}}  &\les   \bigl(\lVert (1+|\nabla_h|^2)\omega\rVert_{\YY_{\mu}}+\lVert (1+ |\nabla_h|^2)\omega\rVert_{S_{\mu}} + \lVert \omega\rVert_{X_{\mu}}\bigr)   \lVert \omega\rVert_{X_{\mu}}   \nn&\indeq   +   \bigl(\lVert (1+|\nabla_h|)\omega\rVert_{\YY_{\mu}}+\lVert (1+   |\nabla_h|)\omega\rVert_{S_{\mu}} \bigr)   \lVert D    \omega\rVert_{X_{\mu}}    \,  ,   \label{8ThswELzXU3X7Ebd1KdZ7v1rN3GiirRXGKWK099ovBM0FDJCvkopYNQ2a128}   \end{align} and   \begin{align}   \sum_{|\aa|=1}&\lVert   D^{\aa}N(s)\rVert_{X_{\mu}} \lesssim    \lVert   \omega\rVert_{X_{\mu}} \sum_{|\aa|=1} \left(      \lVert (1+ |\nabla_h|^{|\alpha_h|+2})\omega\rVert_{\YY_{\mu}}        +\lVert (1 + |\nabla_h|^{|\alpha_h|+2})\omega\rVert_{S_{\mu}}        + \norm{D^{\aa}\omega}_{X_\mu}      \right) \nn&\indeq +        \sum_{|\aa|=1}\lVert D^{\aa} \omega \rVert_{X_{\mu}} \left(        \lVert (1+ |\nabla_h|^{2}) \omega\rVert_{\YY_{\mu}}     +\lVert        (1+ |\nabla_h|^{2}) \omega\rVert_{S_{\mu}}      + \alpha_3        (\norm{z\partial_{z})^{\alpha_3}\omega}_{X_\mu}     \right)        \notag  \\&\indeq    +   \sum_{|\aa|=2}\lVert D^{\aa} \omega \rVert_{X_{\mu}} \left(   \lVert (1+ |\nabla_h|) \omega\rVert_{\YY_{\mu}}     +\lVert (1+   |\nabla_h|) \omega\rVert_{S_{\mu}}      \right) .   \label{8ThswELzXU3X7Ebd1KdZ7v1rN3GiirRXGKWK099ovBM0FDJCvkopYNQ2a127}    \end{align} \end{lemma} \colb \ghoahbkseibhseir We first recall that   $  N_\xi=(\omega\cdot\nabla u -u\cdot\nabla\omega)_\xi(s, z)$. Since the horizontal and vertical components of the vorticity behave differently, we need to consider   them separately. 
This leads to the decomposition
  \begin{align}
  N_\xi=&\left((\omega_h\cdot\nabla_h u_h)_\xi, (\omega_h\cdot\nabla_h u_3)_\xi \right)
  +\left((\omega_3\partial_z u_h)_\xi,  (\omega_3\partial_z u_3)_\xi \right)
    \nn&\quad
  -\left((u_h\cdot\nabla_{h}\omega_h)_\xi , (u_h\cdot\nabla_{h}\omega_3)_\xi \right)
  -\left(\left(\frac{u_3}{z}z\partial_{z}\omega_h\right)_\xi
  ,\left(\frac{u_3}{z}z\partial_{z}\omega_3\right)_\xi\right).
 \label{8ThswELzXU3X7Ebd1KdZ7v1rN3GiirRXGKWK099ovBM0FDJCvkopYNQ2a129} 
  \end{align}
  Thus, in order to prove the lemma, we have to estimate   the above velocity terms. These inequalities are collected in the   next two lemmas. \ghoahbkseibhseir \cole \begin{lemma} \label{L14} Let $\mu \in (0, \mu_0 - \gamma s)$. For the velocity field $u$ and its derivatives, we have \begin{align} \sum_{\xi}\sup_{z\in\Omega_{\mu}}e^{\ee(1+\mu-z)_+|\xi|}w^{\alpha_3}(z)|(  \nabla^{\alpha} u)_\xi|   \lesssim   \lVert|\nabla_h|^{|\alpha|} \omega\rVert_{\YY_{\mu}}+\lVert|\nabla_h|^{|\alpha|} \omega\rVert_{S_{\mu}} + \alpha_3 \norm{\omega}_{X_\mu} , \label{8ThswELzXU3X7Ebd1KdZ7v1rN3GiirRXGKWK099ovBM0FDJCvkopYNQ2a:X:mu:u1} \,   \end{align} \begin{align}   \left\lVert  D^{\alpha} u  \right \rVert_{\xx_{\mu}}     \lesssim   \lVert |\nabla_h|^{|\alpha|}   \omega\rVert_{\YY_{\mu}}+\lVert|\nabla_h|^{|\alpha|}   \omega\rVert_{S_{\mu}} + \alpha_3 \norm{\omega}_{X_\mu} ,   \label{8ThswELzXU3X7Ebd1KdZ7v1rN3GiirRXGKWK099ovBM0FDJCvkopYNQ2a:X:mu:u7}   \,   \end{align} \begin{align}   \left\lVert  D^{\alpha} \nabla_h u   \right \rVert_{\xx_{\mu}}    \lesssim   \lVert|\nabla_h|^{|\alpha|+1}   \omega\rVert_{\YY_{\mu}}+\lVert|\nabla_h|^{|\alpha|+1}   \omega\rVert_{S_{\mu}} +    \alpha_3 \norm{   |\nabla_h|\omega}_{X_\mu},  \label{8ThswELzXU3X7Ebd1KdZ7v1rN3GiirRXGKWK099ovBM0FDJCvkopYNQ2a:X:mu:u3}     \,  \end{align} and  \begin{align}    \sum_{\xi}\sup_{z\in\Omega_{\mu}}e^{\ee(1+\mu-z)_+|\xi|}w^{\alpha_3}(z)|(    D^{\alpha} \partial_{z} u)_\xi|   \lesssim    \lVert|\nabla_h|^{|\alpha|+1}    \omega\rVert_{\YY_{\mu}}+\lVert|\nabla_h|^{|\alpha|+1}    \omega\rVert_{S_{\mu}} + \alpha_3    (\norm{z\partial_{z})^{\alpha_3}\omega}_{X_\mu}    \label{8ThswELzXU3X7Ebd1KdZ7v1rN3GiirRXGKWK099ovBM0FDJCvkopYNQ2a:X:mu:u6}    \,  \end{align} for $0 \leq |\alpha| \leq 1$. \end{lemma} \colb \ghoahbkseibhseir \ghoahbkseibhseir \ghoahbkseibhseir \begin{proof}[Proof of Lemma~\ref{L14}] We first prove \eqref{8ThswELzXU3X7Ebd1KdZ7v1rN3GiirRXGKWK099ovBM0FDJCvkopYNQ2a:X:mu:u1} for $\alpha = (0, 0,0)$. From~\eqref{bio:sav}, we have \begin{equation}   \label{for:u1}   u_1 = \partial_{2}(-\Delta)^{-1}\omega_3 - \partial_{3}(-\Delta)^{-1}\omega_2.   \end{equation} In view of  \eqref{8ThswELzXU3X7Ebd1KdZ7v1rN3GiirRXGKWK099ovBM0FDJCvkopYNQ2a124}, we decompose the first term as   \begin{align}  \notag   (\partial_{2}(-\Delta)^{-1}\omega_3)_{\xi}(s, z)   =& \frac{\ii}{2}\frac{\xi_2}{|\xi|}\Big(\int_0^z e^{-|\xi|(z-\zz)}(1-e^{-2|\xi| \zz})\omega_{3,\xi}(s, \zz) \,d\zz \\&\indeq+   \left(\int_z^{1+\mu}+\int_{1+\mu}^\infty\right) e^{-|\xi|(\zz-z)}(1-e^{-2|\xi| z})\omega_{3,\xi}(s, \zz) \,d\zz\Big) \notag \\=&  I_{1}+I_{2}+I_{3}    \, .    \notag   \end{align} Using the  triangle inequality, we have   \begin{align}   &e^{\ee(1+\mu-z)_+|\xi|}e^{- |z-\zz||\xi|}    \leq   e^{\ee(1+\mu-\zz)_+|\xi|} e^{\ee(\zz-z)_+|\xi|}   e^{- |z-\zz||\xi|}   \le   e^{\ee(1+\mu-\zz)_+|\xi|}e^{-(1-\ee)|\zz-z||\xi|}   \label{8ThswELzXU3X7Ebd1KdZ7v1rN3GiirRXGKWK099ovBM0FDJCvkopYNQ2a133}   \end{align} provided $\epsilon_0\leq 1$. Thus, we further obtain   \begin{align}   e^{\ee(1+\mu-z)_+|\xi|}    (|I_1|+|I_2|)   \les    \int_{0}^{1+\mu}   e^{\ee(1+\mu-z)_+|\xi|}        |\omega_{3,\xi}(s,z)|        \,dz   \les         \lVert e^{\ee(1+\mu-z)_+|\xi|}\omega_{3,    \xi}\rVert_{\SL_{\mu}}    \, .   \label{8ThswELzXU3X7Ebd1KdZ7v1rN3GiirRXGKWK099ovBM0FDJCvkopYNQ2a134}   \end{align} Next using \eqref{8ThswELzXU3X7Ebd1KdZ7v1rN3GiirRXGKWK099ovBM0FDJCvkopYNQ2a133} , we treat the term $I_3$ as   \begin{align} e^{\ee(1+\mu-z)_+|\xi|}|I_{3}|   &\les   \int^{\infty}_{1+\mu} |\omega_{3,\xi}(s, z)| \,dz   \lesssim   \lVert z \omega_{3,\xi}\rVert_{L^2(z \geq 1+\mu)}    \, . \label{8ThswELzXU3X7Ebd1KdZ7v1rN3GiirRXGKWK099ovBM0FDJCvkopYNQ2a136} \end{align} Combining the above estimates gives   \begin{align}   \label{par:2:w:3}   |e^{\ee(1+\mu-z)_+|\xi|}(\partial_{2}(-\Delta)^{-1}\omega_3)_{\xi}(s,   z)| \lesssim   \lVert e^{\ee(1+\mu-z)_+|\xi|}\omega_{3,   \xi}\rVert_{\SL_{\mu}} +    \lVert z \omega_{3,\xi}\rVert_{L^2(z   \geq 1+\mu)}.   \end{align} For the second term in~\eqref{for:u1},   we similarly rewrite the integral as   \begin{align} \notag    (\partial_{3}(-\Delta)^{-1}\omega_2)_{\xi}(z) =&   \frac{1}{2}\Big(-\int_0^z e^{-|\xi|(z-\zz)}(1-e^{-2|\xi|   \zz})\omega_{2,\xi}(s, \zz) \,d\zz    \\&\indeq+   \left(\int_z^{1+\mu}+\int_{1+\mu}^\infty\right)   e^{-|\xi|(\zz-z)}(1+e^{-2|\xi| z})\omega_{2,\xi}(s, \zz) \,d\zz\Big) \notag   \end{align} and obtain the inequality   \begin{align}   \label{par:3:w:2}   |e^{\ee(1+\mu-z)_+|\xi|}(\partial_{3}(-\Delta)^{-1}\omega_2)_{\xi}(z)|   \lesssim   \lVert e^{\ee(1+\mu-z)_+|\xi|}\omega_{2,   \xi}\rVert_{\SL_{\mu}} +    \lVert z \omega_{2,\xi}\rVert_{L^2(z   \geq 1+\mu)}.   \end{align} Summing the bounds \eqref{par:2:w:3} and   \eqref{par:3:w:2} in $\xi$,  we conclude the proof of   \eqref{8ThswELzXU3X7Ebd1KdZ7v1rN3GiirRXGKWK099ovBM0FDJCvkopYNQ2a:X:mu:u1}   for $u_1$ when $|\aa|=0$. By the same procedure, one can check that   the inequality   \eqref{8ThswELzXU3X7Ebd1KdZ7v1rN3GiirRXGKWK099ovBM0FDJCvkopYNQ2a:X:mu:u1}    holds for $u_2$ and $u_3$. \ghoahbkseibhseir The cases $\alpha =   (1,0,0), (0,1,0)$ amount to multiplying the vorticity by $\ii \xi_1$   and $\ii \xi_2$ respectively, and thus the assertion follows by the   same proof as for $\alpha = (0, 0,0)$.  For the case $\alpha=(0,0, 1)$, we apply the  $z$ derivative of  $\partial_{2}(-\Delta)^{-1}\omega_3$ to obtain   \begin{align} \partial_{z}\partial_{2}(-\Delta)^{-1}\omega_{3,\xi}=&   \frac{\ii}{2} \biggl( -\int_0^z  e^{-|\xi|(z-\zz)}(1-e^{-2|\xi| \zz}) \xi_2 \omega_{3,\xi}(s, \zz) \,d\zz  \notag  \\&\indeq +  \int_z^\infty e^{-|\xi|(\zz-z)}(1-e^{-2|\xi| z}) \xi_2 \omega_{3,\xi}(s, \zz) \,d\zz \notag   \\&\indeq + 2  \int_z^\infty  e^{-|\xi|(\zz-z)}e^{-2|\xi| z} \xi_2 \omega_{3, \xi}(s, \zz) \,d\zz\biggr)    \, .   \label{8ThswELzXU3X7Ebd1KdZ7v1rN3GiirRXGKWK099ovBM0FDJCvkopYNQ2a:der:y:u1}   \end{align} The \therere{ LlX efa eN6 v1 B6m3 Ek3J SXUIjX 8P d NKI UFN JvP Ha Vr4T eARP dXEV7B xM 0 A7w 7je p8M 4Q ahOi hEVo} same argument as in the proof   of~\eqref{8ThswELzXU3X7Ebd1KdZ7v1rN3GiirRXGKWK099ovBM0FDJCvkopYNQ2a:X:mu:u1}   for $\aa=(0, 0, 0)$ leads to   \begin{align} \notag   |e^{\ee(1+\mu-z)_+|\xi|}w(z)\partial_{z}(\partial_{2}(-\Delta)^{-1}\omega_3)_{\xi}|    \lesssim   \lVert e^{\ee(1+\mu-z)_+|\xi|}|\xi|\omega_{3,   \xi}\rVert_{\SL_{\mu}} +    \lVert z   |\xi|\omega_{3,\xi}\rVert_{L^2(z \geq 1+\mu)}.   \end{align} Note   that    \begin{align}   \begin{split}   \partial_{z}(\partial_{3}(-\Delta)^{-1}\omega_2)_{\xi}=&   \frac{1}{2}\biggl(\int_0^z |\xi|e^{-|\xi|(z-\zz)}(1-e^{-2|\xi|   \zz})\omega_{2,\xi}(s, \zz) \,d\zz    \\&\indeq+  \int_z^\infty   |\xi| e^{-|\xi|(\zz-z)}(1+e^{-2|\xi| z})\omega_{2,\xi}(s, \zz)   \,d\zz    \\&\indeq+  \int_z^\infty   (-2|\xi|)e^{-|\xi|(\zz-z)}e^{-2|\xi| z}\omega_{2,\xi}(s, \zz)   \,d\zz\biggr)    -\omega_{2,\xi}(z).   \end{split}   \label{8ThswELzXU3X7Ebd1KdZ7v1rN3GiirRXGKWK099ovBM0FDJCvkopYNQ2a1:der:y:u1}   \end{align} Following the arguments for   proving~\eqref{8ThswELzXU3X7Ebd1KdZ7v1rN3GiirRXGKWK099ovBM0FDJCvkopYNQ2a:X:mu:u1}   leads to   \begin{align} \notag   |e^{\ee(1+\mu-z)_+|\xi|}w(z)\partial_{z}(\partial_{3}(-\Delta)^{-1}\omega_2)_{\xi}|   \lesssim&   \lVert e^{\ee(1+\mu-z)_+|\xi|}|\xi|\omega_{2,    \xi}\rVert_{\SL_{\mu}} +    \lVert z   |\xi|\omega_{2,\xi}\rVert_{L^2(z \geq 1+\mu)}\nn&   +   |e^{\ee(1+\mu-z)_+|\xi|}w(z)\omega_{2,\xi}|   \nn\lesssim&   \lVert|\xi| \omega_{2,\xi}\rVert_{\YY_{\mu}}+\lVert|\xi|   \omega_{2,\xi}\rVert_{S_{\mu}} + \alpha_3   \norm{\omega_{2,\xi}}_{\bar\xx_\mu} , \notag   \end{align} which   completes the proof of   \eqref{8ThswELzXU3X7Ebd1KdZ7v1rN3GiirRXGKWK099ovBM0FDJCvkopYNQ2a:X:mu:u1}   for $u_1$ when $|\alpha|=1$. The estimates for $u_2$ and $u_3$   follow similarly.  The presence of the additional factor  $|\xi|$   causes $\omega$ to be replaced by $|\nabla_h| \omega$ in the upper   bounds.  This concludes the proof of \eqref{8ThswELzXU3X7Ebd1KdZ7v1rN3GiirRXGKWK099ovBM0FDJCvkopYNQ2a:X:mu:u1} for $\alpha = (0, 0,1)$. \ghoahbkseibhseir Note that \eqref{8ThswELzXU3X7Ebd1KdZ7v1rN3GiirRXGKWK099ovBM0FDJCvkopYNQ2a:X:mu:u3} is a direct consequence of \eqref{8ThswELzXU3X7Ebd1KdZ7v1rN3GiirRXGKWK099ovBM0FDJCvkopYNQ2a:X:mu:u7}  once replacing $u$ by $\nabla_h u$ in \eqref{8ThswELzXU3X7Ebd1KdZ7v1rN3GiirRXGKWK099ovBM0FDJCvkopYNQ2a:X:mu:u7}. Hence, we only prove~\eqref{8ThswELzXU3X7Ebd1KdZ7v1rN3GiirRXGKWK099ovBM0FDJCvkopYNQ2a:X:mu:u7}. The cases $\aa=(1, 0, 0)$ and $(0, 1, 0)$ of \eqref{8ThswELzXU3X7Ebd1KdZ7v1rN3GiirRXGKWK099ovBM0FDJCvkopYNQ2a:X:mu:u7} are already covered in \eqref{8ThswELzXU3X7Ebd1KdZ7v1rN3GiirRXGKWK099ovBM0FDJCvkopYNQ2a:X:mu:u1} and thus we just need to consider $\aa=(0, 0, 1)$. Noting $z \le w(z)$, this case follows from \eqref{8ThswELzXU3X7Ebd1KdZ7v1rN3GiirRXGKWK099ovBM0FDJCvkopYNQ2a:X:mu:u1}, concluding the proof of \eqref{8ThswELzXU3X7Ebd1KdZ7v1rN3GiirRXGKWK099ovBM0FDJCvkopYNQ2a:X:mu:u7}. \ghoahbkseibhseir In order to prove \eqref{8ThswELzXU3X7Ebd1KdZ7v1rN3GiirRXGKWK099ovBM0FDJCvkopYNQ2a:X:mu:u6}, we need to take the second derivative of the velocity $u$. Again the  cases $\alpha = (1,0,0), (0,1,0)$ are easy since we just need to do the same estimates as in \eqref{8ThswELzXU3X7Ebd1KdZ7v1rN3GiirRXGKWK099ovBM0FDJCvkopYNQ2a:X:mu:u1} with $\xi_i\omega_\xi$ for $i=1,2$. Also the case $z\partial_{z} \nabla_h u$ is essentially covered in \eqref{8ThswELzXU3X7Ebd1KdZ7v1rN3GiirRXGKWK099ovBM0FDJCvkopYNQ2a:X:mu:u3}. Therefore, we only need to consider the second derivative in the vertical direction.  Take the conormal derivative of $\partial_{z}\partial_{2}(-\Delta)^{-1}\omega_{3,\xi}$, we have \begin{align}   \label{co:z:2:3}   z\partial_{z} &\partial_{z}\partial_{2}(-\Delta)^{-1}\omega_{3,\xi}   \nn&= \frac{\ii}{2}z   \biggl( \int_0^z  e^{-|\xi|(z-\zz)}(1-e^{-2|\xi| \zz}) |\xi|\xi_2 \omega_{3,\xi}(s, \zz) \,d\zz   -2 \xi_2 \omega_{3,\xi}(s, z)  \notag  \\&\indeq +  \int_z^\infty  e^{-|\xi|(\zz-z)}(1-e^{-2|\xi| z}) |\xi| \xi_2 \omega_{3,\xi}(s, \zz) \,d\zz + 2 \int_z^\infty e^{-|\xi|(\zz-z)}e^{-2|\xi| z} |\xi|\xi_2 \omega_{3, \xi}(s, \zz)  \,d\zz \notag   \\&\indeq + 2  \int_z^\infty e^{-|\xi|(\zz-z)}e^{-2|\xi| z} |\xi|\xi_2 \omega_{3, \xi}(s, \zz) \,d\zz - 4 \int_z^\infty  e^{-|\xi|(\zz-z)}e^{-2|\xi| z} |\xi|\xi_2 \omega_{3, \xi}(s, \zz) \,d\zz \biggr).   \end{align} Similar estimates as the proof of~\eqref{8ThswELzXU3X7Ebd1KdZ7v1rN3GiirRXGKWK099ovBM0FDJCvkopYNQ2a:X:mu:u1} yield   \begin{align} \notag |e^{\ee(1+\mu-z)_+|\xi|}w(z)z\partial_{z} \partial_{z}\partial_{2}(-\Delta)^{-1}\omega_{3,\xi}| \lesssim   \lVert|\xi|^2 \omega_{3,\xi}\rVert_{\YY_{\mu}}+\lVert|\xi|^2 \omega_{3,\xi}\rVert_{S_{\mu}} + \alpha_3 \norm{\omega_{3,\xi}}_{\xx_\mu}.   \end{align} For the conormal derivative of   $\partial_{z}(\partial_{3}(-\Delta)^{-1}\omega_2)_{\xi}$, we have   \begin{align}   \label{der:y:2:u1}   \begin{split}   z\partial_{z}   &\partial_{z}(\partial_{3}(-\Delta)^{-1}\omega_2)_{\xi}\\&    =   \frac{z}{2}\biggl(-\int_0^z |\xi|^2 e^{-|\xi|(z-\zz)}(1-e^{-2|\xi|   \zz})\omega_{2,\xi}(s, \zz) \,d\zz    +  \int_z^\infty |\xi|^2   e^{-|\xi|(\zz-z)}(1+e^{-2|\xi| z})\omega_{2,\xi}(s, \zz) \,d\zz    \\&\indeq -2\int_z^\infty |\xi|^2 e^{-|\xi|(\zz-z)}e^{-2|\xi| z}\omega_{2,\xi}(s, \zz) \,d\zz  - 2\int_z^\infty |\xi|^2 e^{-|\xi|(\zz-z)}e^{-2|\xi|  z}\omega_{2,\xi}(s, \zz) \,d\zz  \\&\indeq + 4\int_z^\infty |\xi|^2  e^{-|\xi|(\zz-z)}e^{-2|\xi| z}\omega_{2,\xi}(s, \zz) \,d\zz \biggr)  -z\partial_{z}\omega_{2,\xi}(z).   \end{split}   \end{align} Note  that unlike in~\eqref{co:z:2:3}, there is an extra boundary term in  the above expression, which accounts for an  $\norm{z\partial_{z}\omega_{2,\xi}}_{\bar\xx_\mu}$ term in an upper  bound. Therefore, it follows   \begin{align} \notag   |e^{\ee(1+\mu-z)_+|\xi|}w(z)z\partial_{z} \partial_{z}\partial_{3}(-\Delta)^{-1}\omega_{2,\xi}| \lesssim   \lVert|\xi|^2 \omega_{2,\xi}\rVert_{\YY_{\mu}}+\lVert|\xi|^2   \omega_{2,\xi}\rVert_{S_{\mu}} + \alpha_3   \norm{z\partial_{z}\omega_{2,\xi}}_{\bar\xx_\mu} \,,   \end{align}   concluding the proof. \end{proof} \ghoahbkseibhseir \begin{lemma}   \label{} \label{vel:ver}Let $\mu \in (0, \mu_0 - \gamma s)$. For the   vertical velocity $u_3$ and its derivatives,  the bounds   \begin{align} \left\lVert  D^\alpha \left(\frac{ u_3}{z}\right)    \right \rVert_{\xx_{\mu}}    \lesssim      \lVert |\nabla_h|^{|\alpha_h|+1}\omega\rVert_{\YY_{\mu}}    +\lVert |\nabla_h|^{|\alpha_h|+1}\omega\rVert_{S_{\mu}}          ,    \label{8ThswELzXU3X7Ebd1KdZ7v1rN3GiirRXGKWK099ovBM0FDJCvkopYNQ2a:X:mu:u2}    \end{align} \begin{align}  \left\lVert  \frac{    D^{\alpha}\nabla_hu_3}{z}  \right \rVert_{\xx_{\mu}}    \lesssim    \lVert |\nabla_h|^{|\alpha_h|+2}\omega\rVert_{\YY_{\mu}}    +\lVert |\nabla_h|^{|\alpha_h|+2}\omega\rVert_{S_{\mu}},    \label{8ThswELzXU3X7Ebd1KdZ7v1rN3GiirRXGKWK099ovBM0FDJCvkopYNQ2a:X:mu:u5} \end{align} and \begin{align}   \left\lVert  D^\alpha \left(\frac{ \partial u_3}{\partial z}\right)  \right \rVert_{\xx_{\mu}}    \lesssim \lVert |\nabla_h|^{|\alpha_h|+1}\omega\rVert_{\YY_{\mu}} +\lVert |\nabla_h|^{|\alpha_h|+1}\omega\rVert_{S_{\mu}} +\alpha_3 \norm{|\nabla_h|\omega}_{X_\mu}, \label{8ThswELzXU3X7Ebd1KdZ7v1rN3GiirRXGKWK099ovBM0FDJCvkopYNQ2a:X:mu:u4} \end{align} hold for $|\alpha|\le1$. \end{lemma} \ghoahbkseibhseir    \begin{proof}[Proof of Lemma~\ref{vel:ver}]  First we prove   \eqref{8ThswELzXU3X7Ebd1KdZ7v1rN3GiirRXGKWK099ovBM0FDJCvkopYNQ2a:X:mu:u2},   beginning with the case $\alpha = (0, 0,0)$.  Using   \eqref{8ThswELzXU3X7Ebd1KdZ7v1rN3GiirRXGKWK099ovBM0FDJCvkopYNQ2a124}   we decompose the term $u_{3,\xi}/z$ as   \begin{align}   \frac{u_{3,\xi}}{z} =&   \frac{1}{z}(\partial_{1}(-\Delta)^{-1}\omega_{2,\xi} -   \partial_{2}(-\Delta)^{-1}\omega_{1,\xi})   \nn=&   -   \frac{\xi_1}{|\xi|}\frac{\ii}{2z}   \biggl(\int_0^z   e^{-|\xi|(z-\zz)}(1-e^{-2|\xi| \zz})\omega_{2,\xi}(s, \zz) \,d\zz   \notag  \\&\indeq\indeq\indeq\indeq\indeq\indeq+   \left(\int_z^{1+\mu}+\int_{1+\mu}^\infty\right)   e^{-|\xi|(\zz-z)}(1-e^{-2|\xi| z})\omega_{2,\xi}(s, \zz)   \,d\zz\biggr) \notag \\& + \frac{\xi_2}{|\xi|}\frac{\ii}{2z}   \biggl(\int_0^z e^{-|\xi|(z-\zz)}(1-e^{-2|\xi|   \zz})\omega_{1,\xi}(s, \zz) \,d\zz  \notag    \\&\indeq\indeq\indeq\indeq\indeq\indeq+   \left(\int_z^{1+\mu}+\int_{1+\mu}^\infty\right)   e^{-|\xi|(\zz-z)}(1-e^{-2|\xi| z})\omega_{1,\xi}(s, \zz)   \,d\zz\biggr)  \, .   \llabel{Vwb cL DlGK 1ro3 EEyqEA zw 6 sKe Eg2   sFf jz MtrZ 9kbd xNw66c xf t lzD GZh   8ThswELzXU3X7Ebd1KdZ7v1rN3GiirRXGKWK099ovBM0FDJCvkopYNQ2a139}   \end{align} From the bound   \begin{equation}   \left|\frac{1-e^{-2|\xi| \bar z}}{z}\right|   \les   |\xi|    \ \ \ \ \ \text{for}\  \bar z\leq z     \,   ,   \llabel{xQA WQ KkSX jqmm rEpNuG 6P y loq 8hH lSf Ma LXm5     RzEX W4Y1Bq ib 3 UOh Yw9     8ThswELzXU3X7Ebd1KdZ7v1rN3GiirRXGKWK099ovBM0FDJCvkopYNQ2a140}     \end{equation}  we arrive at  \begin{align}   \left|     \frac{u_{3,\xi}}{z} \right|    &\les    \int_0^z e^{-|\xi|(z-\zz)}     |\xi| |\omega_{h, \xi}(s, \zz)| \,d\zz  +     \left(\int_z^{1+\mu}+\int_{1+\mu}^\infty\right) e^{-|\xi|(\zz-z)}     |\xi| |\omega_{h, \xi}(s, \zz)| \,d\zz   \, .  \label{8ThswELzXU3X7Ebd1KdZ7v1rN3GiirRXGKWK099ovBM0FDJCvkopYNQ2a:MJ}    \end{align} Using   \eqref{8ThswELzXU3X7Ebd1KdZ7v1rN3GiirRXGKWK099ovBM0FDJCvkopYNQ2a133}   and similar arguments as   \eqref{8ThswELzXU3X7Ebd1KdZ7v1rN3GiirRXGKWK099ovBM0FDJCvkopYNQ2a134}--\eqref{8ThswELzXU3X7Ebd1KdZ7v1rN3GiirRXGKWK099ovBM0FDJCvkopYNQ2a136},   we  obtain the inequality   \eqref{8ThswELzXU3X7Ebd1KdZ7v1rN3GiirRXGKWK099ovBM0FDJCvkopYNQ2a:X:mu:u2}   for $|\alpha|=0$. The case  $\alpha = (1,0,0), (0,1,0)$ follows by   adding an extra $x$~derivative in the estimates.  \ghoahbkseibhseir   It remains to consider the case  $\alpha =(0,0,1)$.  From the   incompressibility we have  \begin{align}   \label{8ThswELzXU3X7Ebd1KdZ7v1rN3GiirRXGKWK099ovBM0FDJCvkopYNQ2a:der:y:u2}   z\partial_{z}\left(\frac{u_{3,\xi}}{z}\right)  = \partial_z   u_{3,\xi} - \frac{u_{3,\xi}}{z} = - \ii \xi\cdot u_{h,\xi} -   \frac{u_{3,\xi}}{z} \, . \end{align} The bound for the second term on the far right of \eqref{8ThswELzXU3X7Ebd1KdZ7v1rN3GiirRXGKWK099ovBM0FDJCvkopYNQ2a:der:y:u2}  was established in \eqref{8ThswELzXU3X7Ebd1KdZ7v1rN3GiirRXGKWK099ovBM0FDJCvkopYNQ2a:MJ}, whereas the bound for the first term follows from the inequality for $\alpha = (1,0,0), (0,1,0)$ in \eqref{8ThswELzXU3X7Ebd1KdZ7v1rN3GiirRXGKWK099ovBM0FDJCvkopYNQ2a:X:mu:u1}. \ghoahbkseibhseir Inequality~\eqref{8ThswELzXU3X7Ebd1KdZ7v1rN3GiirRXGKWK099ovBM0FDJCvkopYNQ2a:X:mu:u5} is a direct application of inequality~\eqref{8ThswELzXU3X7Ebd1KdZ7v1rN3GiirRXGKWK099ovBM0FDJCvkopYNQ2a:X:mu:u2} with $u_3$ replaced by $\nabla_hu_3$. Noting that $   \partial_z u_{3,\xi}= - \ii \xi\cdot u_{h,\xi}, $, the inequality~\eqref{8ThswELzXU3X7Ebd1KdZ7v1rN3GiirRXGKWK099ovBM0FDJCvkopYNQ2a:X:mu:u4} is essentially covered in the previous lemma. Thus the proof is concluded.   \end{proof} \ghoahbkseibhseir Having established Lemma~\ref{L14} and \ref{vel:ver}, we give the proof of Lemma~\ref{L01}. \begin{proof}[Proof of Lemma~\ref{L01}] In order to prove  \eqref{8ThswELzXU3X7Ebd1KdZ7v1rN3GiirRXGKWK099ovBM0FDJCvkopYNQ2a128}, using \eqref{8ThswELzXU3X7Ebd1KdZ7v1rN3GiirRXGKWK099ovBM0FDJCvkopYNQ2a129} and a triangle inequality leads to   \begin{equation} \notag e^{\ee(1+\mu-z)_+ |\xi|} \leq   e^{\ee(1+\mu-z)_+ |\eta|} e^{\ee(1+\mu-z)_+ |\xi-\eta|}.   \end{equation} By the definition of the $X_{\mu}$ norm and Young's inequality in $\xi$ and $\eta$, it follows that   \begin{align}   &\lVert N(s)\rVert_{X_{\mu}}   \les     \lVert \omega\rVert_{X_{\mu}} \left\lVert \nabla_hu_h   \right \rVert_{\xx_{\mu}}  +   \lVert z\omega_h\rVert_{\xx_{\mu}} \left\lVert  \frac{\nabla_hu_3}{z}  \right \rVert_{\xx_{\mu}}  +  \lVert  \omega_3\rVert_{\xx_{\mu}} \left\lVert   \partial_{z}u_h \right  \rVert_{\bar\xx_{\mu}}   \notag  \\&\indeq   +  \lVert  \omega_3\rVert_{\xx_{\mu}} \left\lVert \partial_{z}u_3   \right  \rVert_{\xx_{\mu}}   +  \lVert  \nabla_{h}\omega\rVert_{X_{\mu}}\left\lVert  u_h  \right  \rVert_{\xx_{\mu}}       +\lVert \nabla_{h}  \omega\rVert_{X_{\mu}}   \left\lVert  u_h  \right \rVert_{\xx_{\mu}}   \notag  \\&\indeq   +\lVert z\partial_{z}\omega\rVert_{X_{\mu}} \left\lVert   \frac{u_3}{z} \right \rVert_{\xx_{\mu}}  	   +\lVert z\partial_{z}  \omega\rVert_{X_{\mu}} \left\lVert   \frac{u_3}{z} \right \rVert_{\xx_{\mu}}     .     \label{8ThswELzXU3X7Ebd1KdZ7v1rN3GiirRXGKWK099ovBM0FDJCvkopYNQ2a137}     \end{align} Here we estimate the vertical and horizontal     components of the vorticity separately. The reason is that     compared with the horizontal component, the vertical component is     much smaller in $L^\infty$ norm.  Hence, in the $X_\mu$ norm     defined in~\eqref{nor:X:mu}, there is a weight function $w(z)$     only for the horizontal components. Therefore, terms like     $\omega_3\omega_h$ are of the size $1/w(z)$ which is still     controlled. Using Lemma~\ref{L14} with $|\alpha| = 0$  we arrive     at   \begin{align}   \lVert N(s)\rVert_{X_{\mu}}   &\les     \bigl(\lVert (1+|\nabla_h|^2)\omega\rVert_{\YY_{\mu}}+\lVert (1+     |\nabla_h|^2)\omega\rVert_{S_{\mu}} + \alpha_3\lVert      \omega\rVert_{X_{\mu}}\bigr)   \lVert \omega\rVert_{X_{\mu}}   \nn&\indeq   +   \bigl(\lVert (1+|\nabla_h|)\omega\rVert_{\YY_{\mu}}+\lVert (1+   |\nabla_h|)\omega\rVert_{S_{\mu}} \bigr)   \lVert D   \omega\rVert_{X_{\mu}}    \,  ,    \llabel{5h6 f6 o8kw 6frZ wg6fIy   XP n ae1 TQJ Mt2 TT fWWf jJrX ilpYGr Ul Q 4uM 7Ds   8ThswELzXU3X7Ebd1KdZ7v1rN3GiirRXGKWK099ovBM0FDJCvkopYNQ2a135}   \end{align} and   \eqref{8ThswELzXU3X7Ebd1KdZ7v1rN3GiirRXGKWK099ovBM0FDJCvkopYNQ2a128}   is established.  Next we analyze the first order derivatives of the   nonlinear term.  By the Leibniz rule, for $|\alpha| = 1$, we have 
  \begin{align}
    D^{\alpha}N_\xi 
   &=
   \left((D^{\alpha}\omega_h\cdot\nabla_h u_h)_\xi +(\omega_h\cdot\nabla_h D^{\alpha}u_h)_\xi
   , (D^{\alpha}\omega_h\cdot\nabla_h u_3)_\xi +(\omega_h\cdot\nabla_h D^{\alpha}u_3)_\xi\right)
    \nn & \quad
    +\left((D^{\alpha}\omega_3\partial_z u_h)_\xi +(\omega_3D^{\alpha}\partial_z u_h)_\xi
    , (D^{\alpha}\omega_3\partial_z u_3)_\xi +(\omega_3D^{\alpha}\partial_z u_3)_\xi \right)
   \nn & \quad-
   ( D^\alpha u_h\cdot\nabla_{h}\omega)_\xi -(u_h\cdot \nabla_h D^\alpha\omega)_\xi
   -\left(D^\alpha\left(\frac{u_3}{z}\right) z\partial_{z}\omega\right)_\xi -\left(\frac{u_3}{z}   D^\alpha(z\partial_{z})\omega\right)_\xi
   \, .
  \llabel{EQ:der:non}
  \end{align} Similarly to \eqref{8ThswELzXU3X7Ebd1KdZ7v1rN3GiirRXGKWK099ovBM0FDJCvkopYNQ2a137}, one gets   \begin{align}   \lVert    &D^{\alpha}     N(s)\rVert_{X_{\mu}}   \nn    &\quad \les  \lVert D^{\alpha}\omega\rVert_{X_{\mu}}   \left\lVert  \nabla_hu_h  \right \rVert_{\xx_{\mu}}   +  \lVert   \omega\rVert_{X_{\mu}}   \left\lVert  \nabla_hD^{\alpha}u_h  \right   \rVert_{\xx_{\mu}}   +   \lVert zD^{\alpha}\omega_h\rVert_{\xx_{\mu}}    \left\lVert   \frac{\nabla_hu_3}{z}  \right \rVert_{\xx_{\mu}}   \notag    \\&\indeq+  \lVert z\omega_h\rVert_{\xx_{\mu}}   \left\lVert   \frac{\nabla_hD^{\alpha}u_3}{z} \right \rVert_{\xx_{\mu}}  +  \lVert   D^{\alpha}\omega_3\rVert_{\xx_{\mu}}   \left\lVert  \partial_{z}u_h   \right \rVert_{\bar\xx_{\mu}}  +  \lVert \omega_3\rVert_{\xx_{\mu}}   \left\lVert  D^{\alpha}\partial_{z}u_h  \right   \rVert_{\bar\xx_{\mu}}     \notag  \\&\indeq+  \lVert D^{\alpha}\omega_3\rVert_{\xx_{\mu}}   \left\lVert  \partial_{z}u_3  \right \rVert_{\xx_{\mu}}  +  \lVert  \omega_3\rVert_{\xx_{\mu}}   \left\lVert  D^{\alpha}\partial_{z}u_3  \right \rVert_{\xx_{\mu}}  +  \lVert \nabla_{h}\omega\rVert_{X_{\mu}}  \left\lVert  D^{\aa}u_h  \right \rVert_{\xx_{\mu}}  \notag  \\&\indeq  +\lVert \nabla_{h} D^\alpha  \omega\rVert_{X_{\mu}}    \left\lVert  u_h  \right \rVert_{\xx_{\mu}}    +\lVert z\partial_{z}\omega\rVert_{X_{\mu}}    \left\lVert  D^{\alpha} \left(\frac{u_3}{z}\right)  \right   \rVert_{\xx_{\mu}}    +\lVert D^\alpha z\partial_{z}   \omega\rVert_{X_{\mu}}    \left\lVert  \frac{u_3}{z}  \right    \rVert_{\xx_{\mu}}    \, .   \label{8ThswELzXU3X7Ebd1KdZ7v1rN3GiirRXGKWK099ovBM0FDJCvkopYNQ2a130}   \end{align} We appeal to the same idea as in the treatment of the   $N_\xi$ case. Note that the term  $\left\lVert   \nabla_hD^{\alpha}u_3/z  \right \rVert_{\xx_{\mu}}$  leads to the   third order derivative in $\omega$. However, as we shall see, the   third derivative only appears in the $Y_\mu $ norm bound. For   \eqref{8ThswELzXU3X7Ebd1KdZ7v1rN3GiirRXGKWK099ovBM0FDJCvkopYNQ2a127},   we apply the estimates from  Lemma~\ref{L14} --~\ref{vel:ver} to   \eqref{8ThswELzXU3X7Ebd1KdZ7v1rN3GiirRXGKWK099ovBM0FDJCvkopYNQ2a130}   to obtain   \begin{align}   \sum_{|\aa|=1}&\lVert D^{\alpha} N(s)\rVert_{X_{\mu}}   \nn&\les    \lVert \omega\rVert_{X_{\mu}} \sum_{|\aa|=1} \left(   \lVert (1+ |\nabla_h|^{|\alpha_h|+2})\omega\rVert_{\YY_{\mu}}   +\lVert (1 + |\nabla_h|^{|\alpha_h|+2})\omega\rVert_{S_{\mu}}   + \alpha_3 \norm{(z\partial_{z})^{\alpha_3}\omega}_{X_\mu} +   \alpha_3 \norm{ |\nabla_h|\omega}_{X_\mu}      \right)\nn&\indeq+    \sum_{|\aa|=1}\lVert D^{\aa} \omega \rVert_{X_{\mu}} \left(   \lVert (1+ |\nabla_h|^{2}) \omega\rVert_{\YY_{\mu}}     +\lVert (1+   |\nabla_h|^{2}) \omega\rVert_{S_{\mu}}      + \alpha_3   (\norm{z\partial_{z})^{\alpha_3}\omega}_{X_\mu}     \right) \notag  \\&\indeq    +   \sum_{|\aa|=2}\lVert D^{\aa} \omega \rVert_{X_{\mu}} \left(   \lVert (1+ |\nabla_h|) \omega\rVert_{\YY_{\mu}}     +\lVert (1+   |\nabla_h|) \omega\rVert_{S_{\mu}}      \right)      \,    ,   \llabel{0Gd SY FUXL zyQZ hVZMn9 am P 9aE Wzk au0 6d ZghM ym3R jfdePG   ln 8 s7x HYC   8ThswELzXU3X7Ebd1KdZ7v1rN3GiirRXGKWK099ovBM0FDJCvkopYNQ2a141}   \end{align} and   \eqref{8ThswELzXU3X7Ebd1KdZ7v1rN3GiirRXGKWK099ovBM0FDJCvkopYNQ2a127}   is proven. \end{proof} \ghoahbkseibhseir Next, we estimate the term $ D^{\aa} (1+ |\nabla_h|) N(s)$ for $0\le |\aa|\le 1$ in the $\YY$ norm.  \cole \begin{lemma} \label{L09} Let $\mu \in (0, \mu_0 - \gamma s)$ be arbitrary.  For the nonlinear term, we have the inequalities \begin{align}   &\lVert (1+ |\nabla_h|) N(s)\rVert_{\YY_{\mu}}  \nn&\indeq \lesssim    (\lVert (1+ |\nabla_h|)\omega\rVert_{\YY_{\mu}} + \lVert (1+ |\nabla_h|)\omega_3\rVert_{\xx_{\mu}})   (\lVert (1+ |\nabla_h|^2) \omega\rVert_{\YY_{\mu}}+\lVert (1+ |\nabla_h|^2)\omega\rVert_{S_{\mu}} )  \nn&\indeq\quad  +  \sum_{|\alpha|=1}\lVert (1+ |\nabla_h|) D^{\aa}  \omega\rVert_{\YY_{\mu}}  \left( \lVert (1+  |\nabla_h|^2)\omega\rVert_{\YY_{\mu}}        +\lVert (1+  |\nabla_h|^2)\omega\rVert_{S_{\mu}} \right)  \label{8ThswELzXU3X7Ebd1KdZ7v1rN3GiirRXGKWK099ovBM0FDJCvkopYNQ2a258}  \end{align} and   \begin{align}   \sum_{|\alpha|=1}&\lVert  D^{\alpha} (1+ |\nabla_h|)    N(s)\rVert_{Y_{\mu}}     \notag  \\&     \les  \sum_{|\alpha|\le1}\lVert D^{\alpha}(1+     |\nabla_h|)\omega\rVert_{Y_{\mu}}   \sum_{|\alpha|\le2}\bigg(     \lVert D^{\aa}(1+|\nabla_h|) \omega\rVert_{\YY_{\mu}}+\lVert     D^{\aa} (1+ |\nabla_h|) \omega\rVert_{S_{\mu}}    \bigg)    \notag     \\&\quad\quad +  \lVert (1+ |\nabla_h|)\omega\rVert_{\xx_{\mu}}      \left( \sum_{|\alpha|\le1} \lVert D^{\aa} (1+     |\nabla_h|)\omega\rVert_{\YY_{\mu}}        +\lVert |\nabla_h| (1+     |\nabla_h|)\omega\rVert_{S_{\mu}}     \right)   \notag     \\&\quad\quad   +  \sum_{|\alpha|\le1}\lVert D^{\alpha}(1+     |\nabla_h|)\omega_3\rVert_{\xx_{\mu}} \Vert (1+ |\nabla_h|)^2     \omega \Vert_{Y_\mu}    \, . \notag   \end{align}\end{lemma} \colb \ghoahbkseibhseir Since the $Y$ norm is not as sensitive as the $X$ norm to the fast growth of the vorticity near the boundary, we use a rougher decomposition than \eqref{8ThswELzXU3X7Ebd1KdZ7v1rN3GiirRXGKWK099ovBM0FDJCvkopYNQ2a129}, which is   \begin{align} \notag   N_\xi=&(\omega_h\cdot\nabla_h u)_\xi +(\omega_3\partial_z u)_\xi -(u_h\cdot\nabla_{h}\omega)_\xi-\left(\frac{u_3}{z}z\partial_{z}\omega\right)_\xi, \end{align} from where we may easily get   \begin{align}   (1+   |\nabla_h|)N_\xi=&(1+|\xi|)\left((\omega_h\cdot\nabla_h u)_\xi    +(\omega_3\partial_z u)_\xi   -(u_h\cdot\nabla_{h}\omega)_\xi-\left(\frac{u_3}{z}z\partial_{z}\omega\right)_\xi   \right).   \label{8ThswELzXU3X7Ebd1KdZ7v1rN3GiirRXGKWK099ovBM0FDJCvkopYNQ2a:non}   \end{align} Before proceeding to the proof of the Lemma~\ref{L09},   we need the following auxiliary lemmas. First, as an obvious   corollary of Lemma~\ref{L14} (replacing $u$ by $(1+|\nabla_h|)u$ in   \eqref{8ThswELzXU3X7Ebd1KdZ7v1rN3GiirRXGKWK099ovBM0FDJCvkopYNQ2a:X:mu:u7}   and   \eqref{8ThswELzXU3X7Ebd1KdZ7v1rN3GiirRXGKWK099ovBM0FDJCvkopYNQ2a:X:mu:u3}),   we state the bounds for the horizontal velocity   $u_h$. \ghoahbkseibhseir \cole \begin{lemma} \label{L:vel:hor} Let   $\mu \in (0, \mu_0 - \gamma s)$. For the horizontal velocity $u_h$   and its derivatives, we have  \begin{align}    \lVert (1+ |\nabla_h|) \nabla_h^{\alpha_h} u\rVert_{\xx_{\mu}}    \lesssim   \lVert|\nabla_h|^{|\alpha_h|}(1+ |\nabla_h|)     \omega\rVert_{\YY_{\mu}}+\lVert|\nabla_h|^{|\alpha_h|} (1+    |\nabla_h|)\omega\rVert_{S_{\mu}}    \llabel{IV9 Hw Ka6v EjH5    J8Ipr7 Nk C xWR 84T Wnq s0 fsiP qGgs Id1fs5 3A T 71q RIc    8ThswELzXU3X7Ebd1KdZ7v1rN3GiirRXGKWK099ovBM0FDJCvkopYNQ2a:Y:mu:u1}    \,   \end{align} and \begin{align}    \lVert D^{\alpha} \nabla_h    (1+ |\nabla_h|) u \rVert_{\xx_{\mu}}   &\lesssim    \lVert|\nabla_h|^{|\alpha|+1} (1+    |\nabla_h|)\omega\rVert_{\YY_{\mu}}+\lVert|\nabla_h|^{|\alpha|+1}    (1+ |\nabla_h|)\omega\rVert_{S_{\mu}}    \nn&\indeq    + \alpha_3 \norm{ |\nabla_h|(1+ |\nabla_h|)\omega}_{X_\mu}    \llabel{zPX 77 Si23 GirL 9MQZ4F pi g dru NYt h1K 4M Zilv rRk6    B4W5B8 Id 3 Xq9 nhx    8ThswELzXU3X7Ebd1KdZ7v1rN3GiirRXGKWK099ovBM0FDJCvkopYNQ2a:Y:mu:u3}    \,  \end{align} for $0 \leq |\alpha| \leq 1$. \end{lemma} \colb    Similarly as Lemma~\ref{vel:ver}, we get the following    estimates. \begin{lemma} \label{L:vel:ver} Let $\mu \in (0, \mu_0 -     \gamma s)$. For the vertical velocity $u_3$ and its derivatives,    the bound \begin{align}  \left\lVert D^\alpha (1+ |\nabla_h|)\left(\frac{ u_3}{z}\right)  \right\rVert_{\xx_{\mu}}    \lesssim      \lVert  |\nabla_h|^{|\alpha_h|+1}(1+ |\nabla_h|)\omega\rVert_{\YY_{\mu}}  +\lVert |\nabla_h|^{|\alpha_h|+1}(1+  |\nabla_h|)\omega\rVert_{S_{\mu}}   \llabel{EN4 P6 ipZl a2UQ Qx8mda  g7 r VD3 zdD rhB vk LDJo tKyV 5IrmyJ R5 e txS 1cv  8ThswELzXU3X7Ebd1KdZ7v1rN3GiirRXGKWK099ovBM0FDJCvkopYNQ2a:Y:mu:u2}  \end{align} holds for $|\alpha|\le1$. \end{lemma} We also need an  estimate for the $Y_\mu$ norm of the velocity. \begin{lemma}  \label{Y:vel:hor} Let $\mu \in (0, \mu_0 - \gamma s)$. For the velocity $u$ and its derivatives,  the estimate \begin{align} \Vert D^{\alpha}\partial_{z}(1+ |\nabla_h|)u \Vert_{Y_\mu}   &\lesssim      \lVert |\nabla_h|^{|\alpha|}(1+ |\nabla_h|)\omega\rVert_{\YY_{\mu}}        +\lVert |\nabla_h|^{|\alpha|}(1+ |\nabla_h|)\omega\rVert_{S_{\mu}}  \nn&\indeq+   \alpha_3\Vert z\partial_{z} (1+ |\nabla_h|)\omega \Vert_{Y_\mu}   \llabel{EsY xG zj2T rfSR myZo4L m5 D mqN iZd acg GQ 0KRw QKGX g9o8v8 wm B fUu tCO 8ThswELzXU3X7Ebd1KdZ7v1rN3GiirRXGKWK099ovBM0FDJCvkopYNQ2a:Y:mu:u4} \end{align} holds for $|\alpha|\le1$. \end{lemma}   \begin{proof}[Proof of Lemma~\ref{Y:vel:hor}] Since the estimates for $u_1$ and $u_2$ are similar, we only focus on the first component. Similarly to the proof of Lemma~\ref{L14}, from equation~\eqref{8ThswELzXU3X7Ebd1KdZ7v1rN3GiirRXGKWK099ovBM0FDJCvkopYNQ2a:der:y:u1}, we arrive at   \begin{align} \notag   \Vert D^{\alpha}(1+   |\nabla_h|)\partial_{z}(\partial_{2}(-\Delta)^{-1}\omega_3)_{\xi}   \Vert_{Y_\mu}     \lesssim   \Vert D^{\alpha}(1+ |\nabla_h|)\omega_3   \Vert_{Y_\mu}  +    \lVert zD^{\alpha}(1+ |\nabla_h|)\omega_3   \rVert_{L^2(z \geq 1+\mu)}   \end{align} for $0\le|\alpha|\le1$ with   $\alpha_3=0$, where we also used   \begin{align} \notag   \Vert   e^{|\xi|z}\Vert_{L^{1}} =     |\xi|^{-1}   \end{align} for $|\xi|\neq0.$ While for $\alpha_3=1$, we have the   representation formula   \begin{align} \notag   z\partial_{z}   &\partial_{z}\partial_{2}(-\Delta)^{-1}\omega_{3,\xi}   \nn&=   \frac{\ii}{2}z   \biggl( \int_0^z  e^{-|\xi|(z-\zz)}(1-e^{-2|\xi|    \zz}) |\xi|\xi_2 \omega_{3,\xi}(s, \zz) \,d\zz   -2 \xi_2   \omega_{3,\xi}(s, z)  \notag  \\&\indeq +  \int_z^\infty   e^{-|\xi|(\zz-z)}(1-e^{-2|\xi| z}) |\xi| \xi_2 \omega_{3,\xi}(s,   \zz) \,d\zz + 2 \int_z^\infty  e^{-|\xi|(\zz-z)}e^{-2|\xi| z} |\xi|\xi_2 \omega_{3, \xi}(s, \zz) \,d\zz \notag   \\&\indeq + 2  \int_z^\infty  e^{-|\xi|(\zz-z)}e^{-2|\xi| z} |\xi|\xi_2 \omega_{3, \xi}(s, \zz) \,d\zz - 4 \int_z^\infty e^{-|\xi|(\zz-z)}e^{-2|\xi| z} |\xi|\xi_2 \omega_{3, \xi}(s, \zz) \,d\zz \biggr), \notag   \end{align} from where we obtain \begin{align} \notag   \Vert z\partial_{z}(1+ |\nabla_h|)\partial_{z}(\partial_{2}(-\Delta)^{-1}\omega_3)_{\xi} \Vert_{Y_\mu}     \lesssim   \Vert |\nabla_h|(1+ |\nabla_h|)\omega_3 \Vert_{Y_\mu}  +    \lVert z |\nabla_h| (1+ |\nabla_h|)\omega_3 \rVert_{L^2(z \geq 1+\mu)}.   \end{align}  Next we consider the term   $\partial_{3}(-\Delta)^{-1}\omega_2$, which is slightly different   since there is a vorticity term in the $z$ derivative.  Note that   \begin{align}   \begin{split}    \partial_{z}(\partial_{3}(-\Delta)^{-1}\omega_2)_{\xi}=&   \frac{1}{2}\biggl(\int_0^z |\xi|e^{-|\xi|(z-\zz)}(1-e^{-2|\xi|   \zz})\omega_{2,\xi}(s, \zz) \,d\zz    \\&\indeq+  \int_z^\infty   |\xi| e^{-|\xi|(\zz-z)}(1+e^{-2|\xi| z})\omega_{2,\xi}(s, \zz)   \,d\zz    \\&\indeq+  \int_z^\infty   (-2|\xi|)e^{-|\xi|(\zz-z)}e^{-2|\xi| z}\omega_{2,\xi}(s, \zz)   \,d\zz\biggr)    -\omega_{2,\xi}(z).   \end{split} \notag   \end{align} It follows that   \begin{align}   \Vert D^{\alpha}(1+ |\nabla_h|)\partial_{z}(\partial_{3}(-\Delta)^{-1}\omega_2)_{\xi} \Vert_{Y_\mu}     \lesssim   \Vert D^{\alpha}(1+ |\nabla_h|)\omega_2 \Vert_{Y_\mu}  +    \lVert zD^{\alpha}(1+ |\nabla_h|)\omega_2 \rVert_{L^2(z \geq 1+\mu)} \notag   \end{align} for $0\le|\alpha|\le1$ with $\alpha_3=0$. At last, from \eqref{der:y:2:u1}, we obtain   \begin{align}   &\Vert D^{\alpha}(1+   |\nabla_h|)\partial_{z}(\partial_{3}(-\Delta)^{-1}\omega_2)_{\xi}    \Vert_{Y_\mu}     \lesssim   \Vert D^{\alpha}(1+ |\nabla_h|)\omega_2   \Vert_{Y_\mu}  +  \Vert |\nabla_h|(1+ |\nabla_h|)\omega_2   \Vert_{Y_\mu}    \nn&\indeq +    \lVert z |\nabla_h| (1+   |\nabla_h|)\omega_2 \rVert_{L^2(z \geq 1+\mu)} \notag   \end{align}   when $\alpha_3=1$. Note that every term in $u_3$ appears in either   $u_1$ or $u_2$. Hence the estimate of $u_3$ is essentially covered   in the previous cases concluding the proof.   \end{proof} Now we are   ready to prove Lemma~\ref{L09}. \begin{proof}[Proof of Lemma~\ref{L09}] By writing the nonlinear term as \eqref{8ThswELzXU3X7Ebd1KdZ7v1rN3GiirRXGKWK099ovBM0FDJCvkopYNQ2a:non} and using the definition of the $\YY_\mu$ norm,  we obtain, similarly to \eqref{8ThswELzXU3X7Ebd1KdZ7v1rN3GiirRXGKWK099ovBM0FDJCvkopYNQ2a137}, \begin{align*}   &\lVert (1+ |\nabla_h|) N(s)\rVert_{\YY_{\mu}} \nn&\indeq   \les \lVert (1+ |\nabla_h|)\omega_h\rVert_{\YY_{\mu}} \lVert (1+ |\nabla_h|)\nabla_hu\rVert_{\xx_{\mu}}    + \lVert (1+  |\nabla_h|)\omega_3\rVert_{\xx_{\mu}}    \lVert ((1+ |\nabla_h|)\partial_{z}u)_\xi \rVert_{\YY_{\mu}}    \nn&\indeq\quad + \lVert (1+ |\nabla_h|)\nabla_h\omega\rVert_{\YY_{\mu}}   \lVert (1+ |\nabla_h|)u_h\rVert_{\xx_{\mu}}   + \lVert (1+ |\nabla_h|)z\partial_{z}\omega\rVert_{\YY_{\mu}}  \left  \lVert (1+ |\nabla_h|) \left(\frac{u_3}{z}\right)  \right\rVert_{\xx_{\mu}}     \, .   \end{align*} By  Lemma~\ref{L:vel:hor}--~\ref{Y:vel:hor}, it follows   \begin{align*}  &\lVert (1+ |\nabla_h|) N(s)\rVert_{\YY_{\mu}}   \nn&\quad   \les  \lVert (1+ |\nabla_h|)\omega_h\rVert_{\YY_{\mu}}  (\lVert|\nabla_h|(1+ |\nabla_h|)  \omega\rVert_{\YY_{\mu}}+\lVert|\nabla_h| (1+  |\nabla_h|)\omega\rVert_{S_{\mu}} )   \nn&\indeq\quad   + \lVert (1+ |\nabla_h|)\omega_3\rVert_{\xx_{\mu}}    \left(\lVert (1+ |\nabla_h|)\omega\rVert_{\YY_{\mu}}        +\lVert   (1+ |\nabla_h|)\omega\rVert_{S_{\mu}}\right)   \nn&\indeq\quad   +   \lVert (1+ |\nabla_h|)\nabla_h\omega\rVert_{\YY_{\mu}}   \left(   \lVert(1+ |\nabla_h|) \omega\rVert_{\YY_{\mu}}+\lVert (1+    |\nabla_h|)\omega\rVert_{S_{\mu}}  \right)   \nn&\indeq\quad  +   \lVert (1+ |\nabla_h|)z\partial_{z}\omega\rVert_{\YY_{\mu}}  \left(   \lVert |\nabla_h|(1+ |\nabla_h|)\omega\rVert_{\YY_{\mu}}   +\lVert |\nabla_h|(1+ |\nabla_h|)\omega\rVert_{S_{\mu}} \right).   \end{align*} Reorganizing the terms on the right side further gives   \begin{align*} \notag   &\lVert (1+ |\nabla_h|) N(s)\rVert_{\YY_{\mu}}    \nn&\quad \les  (\lVert (1+ |\nabla_h|)\omega\rVert_{\YY_{\mu}} + \lVert (1+ |\nabla_h|)\omega_3\rVert_{\xx_{\mu}})   (\lVert (1+ |\nabla_h|^2) \omega\rVert_{\YY_{\mu}}+\lVert (1+ |\nabla_h|^2)\omega\rVert_{S_{\mu}} )  \nn&\indeq\quad  + \sum_{|\alpha|=1}\lVert (1+ |\nabla_h|) D^{\aa} \omega\rVert_{\YY_{\mu}}  \left( \lVert (1+ |\nabla_h|^2)\omega\rVert_{\YY_{\mu}}        +\lVert (1+ |\nabla_h|^2)\omega\rVert_{S_{\mu}} \right)    \, . \notag   \end{align*} Using the bounds in Lemma~\ref{L14} with $|\aa| = 0$,   we arrive at    \eqref{8ThswELzXU3X7Ebd1KdZ7v1rN3GiirRXGKWK099ovBM0FDJCvkopYNQ2a258}. We   next consider the case $|\alpha|=1$. In view of   \eqref{8ThswELzXU3X7Ebd1KdZ7v1rN3GiirRXGKWK099ovBM0FDJCvkopYNQ2a:non},   we use the Leibniz rule to obtain   \begin{align}     D^{\alpha}(1+   |\nabla_h|)N_\xi     =&   (1+|\xi|)\bigg((D^{\alpha}\omega_h\cdot\nabla_h u)_\xi   +(\omega_h\cdot\nabla_h D^{\alpha}u)_\xi   +(D^{\alpha}\omega_3\partial_z u)_\xi      \nn &   \quad+(\omega_3D^{\alpha}\partial_z u)_\xi    -    ( D^\alpha   u_h\cdot\nabla_{h}\omega)_\xi -(u_h\cdot \nabla_h   D^\alpha\omega)_\xi    \nn & \quad    -\left(D^\alpha\left(\frac{u_3}{z}\right)    z\partial_{z}\omega\right)_\xi -\left(\frac{u_3}{z}    D^\alpha(z\partial_{z})\omega\right)_\xi\bigg)    \, .    \llabel{cKc zz kx4U fhuA a8pYzW Vq 9 Sp6 CmA cZL Mx ceBX Dwug    sjWuii Gl v JDb 08h     8ThswELzXU3X7Ebd1KdZ7v1rN3GiirRXGKWK099ovBM0FDJCvkopYNQ2a:der:non}    \end{align} From a triangle inequality we have    \begin{equation}    \notag   e^{\ee(1+\mu-z)_+ |\xi|} \leq   e^{\ee(1+\mu-z)_+ |\eta|}    e^{\ee(1+\mu-z)_+ |\xi-\eta|}.   \end{equation}  Therefore, by the    definition of the $Y_{\mu}$ norm and Young's inequality in $\xi$    and $\eta$, it follows    \begin{align}   &\lVert  D^{\alpha} (1+ |\nabla_h|)    N(s)\rVert_{Y_{\mu}}     \les   \lVert D^{\alpha}(1+ |\nabla_h|)\omega\rVert_{Y_{\mu}}   \left\lVert   \nabla_h(1+ |\nabla_h|)u  \right \rVert_{\xx_{\mu}}        \notag   \\&\quad+  \lVert (1+ |\nabla_h|)\omega\rVert_{Y_{\mu}}   \left\lVert  \nabla_hD^{\alpha}(1+ |\nabla_h|)u  \right   \rVert_{\xx_{\mu}}  +  \lVert D^{\alpha}(1+   |\nabla_h|)\omega_3\rVert_{\xx_{\mu}} \Vert \partial_{z}(1+   |\nabla_h|)u \Vert_{Y_\mu}       \notag  \\&\quad+  \lVert (1+ |\nabla_h|)\omega_3\rVert_{\xx_{\mu}} \Vert  D^{\alpha}\partial_{z}(1+ |\nabla_h|)u \Vert_{Y_\mu}  +  \lVert   \nabla_{h}(1+ |\nabla_h|)\omega\rVert_{Y_{\mu}}  \left\lVert  D^{\aa}  (1+ |\nabla_h|)u_h  \right \rVert_{\xx_{\mu}}        \notag  \\&\quad  +\lVert \nabla_{h}D^\alpha  (1+ |\nabla_h|)\omega\rVert_{Y_{\mu}}  \left\lVert  (1+ |\nabla_h|)u_h  \right \rVert_{\xx_{\mu}}    +\lVert  z\partial_{z}(1+ |\nabla_h|)\omega\rVert_{Y_{\mu}}    \left\lVert  D^{\alpha} \left((1+ |\nabla_h|)\frac{u_3}{z}\right)  \right  \rVert_{\xx_{\mu}}   \notag  \\&\quad   +\lVert D^\alpha  z\partial_{z}  (1+ |\nabla_h|)\omega\rVert_{Y_{\mu}}    \left\lVert  (1+ |\nabla_h|)\frac{u_3}{z}  \right \rVert_{\xx_{\mu}}    \, . \notag   \end{align} Using Lemmas~\ref{L14}, ~\ref{L:vel:hor},   ~\ref{L:vel:ver}, and~\ref{Y:vel:hor} gives   \begin{align}   \sum_{|\alpha|=1}&\lVert     D^{\alpha} (1+ |\nabla_h|)   N(s)\rVert_{Y_{\mu}}     \les  \sum_{|\alpha|\le1}\lVert   D^{\alpha}(1+ |\nabla_h|)\omega\rVert_{Y_{\mu}}       \notag  \\&  \times\bigg( \sum_{|\alpha|=1}\lVert(1+|\nabla_h|^{2}) D^{\aa} \omega\rVert_{\YY_{\mu}}+\lVert (1+|\nabla_h|^{2})(1+ |\nabla_h|) \omega\rVert_{S_{\mu}}    +     \norm{ (1+ |\nabla_h|) \omega}_{X_\mu}  \bigg)   \notag  \\&   +  \lVert (1+ |\nabla_h|)\omega\rVert_{Y_{\mu}}   \bigg(   \lVert|\nabla_h|^{2} (1+    |\nabla_h|)\omega\rVert_{\YY_{\mu}}+\lVert|\nabla_h|^{2} (1+   |\nabla_h|)\omega\rVert_{S_{\mu}}     \notag  \\&   + \norm{   |\nabla_h|(1+ |\nabla_h|)\omega}_{X_\mu} + \sum_{|\alpha|=1} \lVert   D^\alpha z\partial_{z}  (1+ |\nabla_h|)\omega\rVert_{Y_{\mu}} \bigg)   \notag  \\& +  \lVert (1+ |\nabla_h|)\omega_3\rVert_{\xx_{\mu}}   \left( \sum_{|\alpha|=1} \lVert D^{\aa} (1+   |\nabla_h|)\omega\rVert_{\YY_{\mu}}        +\lVert |\nabla_h| (1+   |\nabla_h|)\omega\rVert_{S_{\mu}}     \right)   \notag  \\&   +  \sum_{|\alpha|=1}\lVert D^{\alpha}(1+  |\nabla_h|)\omega_3\rVert_{\xx_{\mu}} \Vert (1+ |\nabla_h|)^2 \omega  \Vert_{Y_\mu}    \,  \notag   \end{align} concluding the  proof. \end{proof} \ghoahbkseibhseir Finally we consider the Sobolev  norm estimates for the nonlinear term. \cole \begin{lemma} \label{L04} Let $\mu \in (0, \mu_0 - \gamma s)$ be arbitrary.  We have \begin{align} \label{8ThswELzXU3X7Ebd1KdZ7v1rN3GiirRXGKWK099ovBM0FDJCvkopYNQ2a164}  \lVert (1+& |\nabla_h|) N(s)\rVert_{S_\mu}  \nn& \lesssim \sum_{|\alpha|\le1}\left( \norm{D^\aa (1+ |\nabla_h|) \omega}_{\YY_\mu} + \norm{\nabla^\aa(1+ |\nabla_h|)\omega}_{S_\mu} \right)   \sum_{|\aa|\le1}   \lVert \nabla^\aa (1+ |\nabla_h|)\omega\rVert_{S_\mu}   \end{align} and     \begin{align} \sum_{|\aa|=1}  &\lVert \nabla^{\aa}(1+ |\nabla_h|) N (s)\rVert_{S_\mu} \nn & \les \sum_{|\aa|\le2}\norm{\nabla^{\aa} (1+ |\nabla_h|)\omega}_{S_\mu}  \sum_{|\aa|\le2}(\norm{D^\aa (1+ |\nabla_h|)\omega}_{\YY_\mu} + \norm{\nabla^\aa (1+ |\nabla_h|)\omega}_{S_\mu}  ). \label{8ThswELzXU3X7Ebd1KdZ7v1rN3GiirRXGKWK099ovBM0FDJCvkopYNQ2a155} \end{align} \end{lemma} \colb \ghoahbkseibhseir \begin{proof} We first note that $ z (\omega \cdot \nabla u-u\cdot \nabla \omega) = z\omega\cdot \nabla u-u\cdot z \nabla \omega $. To prove the inequality~\eqref{8ThswELzXU3X7Ebd1KdZ7v1rN3GiirRXGKWK099ovBM0FDJCvkopYNQ2a164}, we first need several estimates for the velocity.  From H\"older's  inequality in $z$ and Young's inequality in $\xi$ we deduce \begin{align} \sum_\xi   \lVert u_{\xi}\rVert_{L^\infty( z\ge1+\mu  )} \les \sum_\xi  \int_0^\infty |\omega_\xi (z)| \, dz   \les \norm{\omega}_{\YY_\mu} + \norm{\omega}_{S_\mu}  \, . \label{l:inf:u} \end{align} The estimates of the horizontal derivatives $\nabla_h u$ are obtained similarly with an extra $|\xi|$. In view of Lemma~\ref{der:str} and the equations~\eqref{8ThswELzXU3X7Ebd1KdZ7v1rN3GiirRXGKWK099ovBM0FDJCvkopYNQ2a:der:y:u1} with \eqref{8ThswELzXU3X7Ebd1KdZ7v1rN3GiirRXGKWK099ovBM0FDJCvkopYNQ2a1:der:y:u1}, there is a term $\omega$ coming from taking the $z$ derivative of the velocity. Using again H\"older's inequality and Young's inequality, we arrive at \begin{align} \sum_\xi   \lVert (\partial_{z} u)_{\xi}\rVert_{L^\infty( z\ge1+\mu  )}&  \les \sum_\xi  \left( \int_0^\infty |\xi\omega_\xi (z)| \, dz  + \lVert \omega_{\xi}\rVert_{L^\infty( z\ge1+\mu  )}\right) \les \norm{\nabla_h  \omega}_{\YY_\mu} + \norm{\nabla \omega}_{S_\mu}  \,  \notag \end{align} which together with~\eqref{l:inf:u} imply \begin{align} \sum_\xi   \lVert (\nabla u)_{\xi}\rVert_{L^\infty( z\ge1+\mu  )}&  \les \sum_\xi  \left( \int_0^\infty |\xi\omega_\xi (z)| \, dz  + \lVert \omega_{\xi}\rVert_{L^\infty( z\ge1+\mu  )}\right) \les \norm{\nabla_h \omega}_{\YY_\mu} + \norm{\nabla \omega}_{S_\mu}  \, . \label{l:inf:der:u1} \end{align} Applying the above estimates about the velocity to the nonlinearity leads to   \begin{align}  \lVert N(s)\rVert_{S_\mu}   \lesssim   \sum_{|\alpha|\le1}\left( \norm{ D^\aa \omega}_{\YY_\mu} + \norm{\nabla^\aa\omega}_{S_\mu}  \right) \sum_{|\aa|\le1}   \lVert \nabla^\aa \omega\rVert_{S_\mu}. \notag   \end{align} Using the triangle inequality  $   |\xi| \le |\xi-\eta|   + |\eta| $, we further obtain   \begin{align}  \lVert (1+&   |\nabla_h|) N(s)\rVert_{S_\mu} \lesssim   \sum_{|\alpha|\le1}\left(   \norm{ D^\aa (1+ |\nabla_h|) \omega}_{\YY_\mu} + \norm{\nabla^\aa(1+   |\nabla_h|)\omega}_{S_\mu}  \right)   \sum_{|\aa|\le1}   \lVert   \nabla^\aa (1+ |\nabla_h|)\omega\rVert_{S_\mu}. \notag    \end{align} Next we consider   \eqref{8ThswELzXU3X7Ebd1KdZ7v1rN3GiirRXGKWK099ovBM0FDJCvkopYNQ2a155}. When   $|\aa|=1$, by the Leibniz rule we have    \begin{align}   z\nabla^{\alpha}N_\xi     &=    (z\nabla^{\alpha}\omega\cdot\nabla   u)_\xi +(z\omega\cdot \nabla^{\alpha} \nabla u)_\xi     -    (   \nabla^\alpha u\cdot z\nabla\omega)_\xi -(u\cdot   z\nabla^\alpha\nabla\omega)_\xi    \, .   \llabel{BOV C1 pni6 4TTq   Opzezq ZB J y5o KS8 BhH sd nKkH gnZl UCm7j0 Iv Y jQE 7JN   8ThswELzXU3X7Ebd1KdZ7v1rN3GiirRXGKWK099ovBM0FDJCvkopYNQ2a:der:non}   \end{align} Therefore, from H\"older's inequality in $z$ and Young's inequality in $\xi$ we deduce \begin{align} \label{l:inf:der:non1} \norm{\nabla^{\aa} N(s)}_{S_\mu} &\les \norm{\nabla^{\aa} \omega}_{S_\mu} \sum_\xi \norm{ (\nabla u)_{\xi}}_{L^\infty(z\geq 1+\mu)} + \norm{ \omega}_{S_\mu} \sum_\xi \norm{(\nabla^{\aa}\nabla u)_{\xi}}_{L^\infty(z\geq 1+\mu)} \notag\\ &\indeq   + \norm{\nabla \omega}_{S_\mu} \sum_\xi \norm{ (\nabla^{\aa} u)_{\xi}}_{L^\infty(z\geq 1+\mu)} +  \norm{\nabla^{\aa}\nabla \omega}_{S_\mu} \sum_\xi \norm{u_{\xi}}_{L^\infty(z\geq 1+\mu)}   \, . \notag \end{align} Having both~\eqref{l:inf:u} and \eqref{l:inf:der:u1} established, we only need to think about the term $\norm{(\nabla^{\aa}\nabla u)_{\xi}}_{L^\infty(z\geq 1+\mu)} $. Note that the $L^\infty$ norms of both $\nabla_h\nabla u$ and $\nabla \nabla_h u$ are both essentially covered in the treatment for  $\sum_\xi   \lVert (\nabla u)_{\xi}\rVert_{L^\infty( z\ge1+\mu  )}$ (by adding an extra $|\xi|$). So we focus on the term $\partial_{z}\partial_{z}u$. From Lemma~\ref{der:str} and the equations~\eqref{8ThswELzXU3X7Ebd1KdZ7v1rN3GiirRXGKWK099ovBM0FDJCvkopYNQ2a:der:y:u1} with \eqref{8ThswELzXU3X7Ebd1KdZ7v1rN3GiirRXGKWK099ovBM0FDJCvkopYNQ2a1:der:y:u1}, we have  \begin{align} \sum_\xi   \lVert (\partial_{z}\partial_{z} u)_{\xi}\rVert_{L^\infty( z\ge1+\mu  )}&  \les \sum_\xi  \left( \int_0^\infty ||\xi|^2\omega_\xi  (z)| \, dz  + \lVert \partial_{z}\omega_{\xi}\rVert_{L^\infty( z\ge1+\mu  )}\right) \nn& \les \norm{\nabla_h^2 \omega}_{\YY_\mu} + \norm{\nabla^2 \omega}_{S_\mu} . \,  \end{align} In summary, we obtain \begin{align} \label{l:inf:der:u2} \sum_\xi   \lVert (\nabla^2 u)_{\xi}\rVert_{L^\infty( z\ge1+\mu  )}&  \les \sum_\xi  \left( \int_0^\infty ||\xi|^2\omega_\xi (z)| \, dz  + \lVert \partial_{z}\omega_{\xi}\rVert_{L^\infty( z\ge1+\mu  )}\right) \nn& \les \norm{\nabla_h^2 \omega}_{\YY_\mu} + \norm{\nabla^2 \omega}_{S_\mu} . \,  \end{align} Combining~\eqref{l:inf:u}, \eqref{l:inf:der:u1}, \eqref{l:inf:der:u2}, and \eqref{l:inf:der:non1} gives \begin{align} \norm{\nabla^{\aa} N(s)}_{S_\mu}  \les \sum_{|\aa|\le2}\norm{\nabla^{\aa} \omega}_{S_\mu} \sum_{|\aa|\le2}(\norm{D^\aa \omega}_{\YY_\mu} + \norm{\nabla^\aa \omega}_{S_\mu}  ) \, . \notag \end{align} The inequality~\eqref{8ThswELzXU3X7Ebd1KdZ7v1rN3GiirRXGKWK099ovBM0FDJCvkopYNQ2a155} follows by applying a triangle inequality and repeating the above process, and thus the details are omitted. \end{proof} \ghoahbkseibhseir \ghoahbkseibhseir \section{Main estimates for the analytic norms} \label{X:Y:est} \ghoahbkseibhseir In  this section,  we give two lemmas concerning the estimates of the analytic norms in term of the analytic with Sobolev norm of the nonlinearity and only comment about the proof at the end.  The first lemma gives the  $\XX_\mu$-norm estimate. \begin{lemma}[\bf Main $X$ norm estimate] \label{lem:main:X} \cole The nonlinear term in \eqref{kernel:est} is bounded in the $X_\mu$ norm as \begin{align} &(\mu_0-\mu-\gamma s) \sum_{|\aa|= 2} \left\lVert D^{\aa}\int_0^\infty G(t-s, z, \zz)N(s, \zz) \,d\zz\right\rVert_{X_\mu}  \notag\\ &\quad  + \sum_{|\aa| \leq 1} \left\lVert D^{\aa}\int_0^\infty G(t-s, z, \zz)N(s, \zz) \,d\zz\right\rVert_{X_{\mu_1}}  \notag\\ &\quad \quad  \les \sum_{|\aa| \leq 1}\lVert D^{\aa} N(s)\rVert_{X_{\mu_2}}    + \frac{1}{(\mu_0-\mu-\gamma s)^{1/2}}   \sum_{|\aa| \leq 1} \lVert \nabla^{\aa} N(s)\rVert_{S_{\mu_2}}    \, \label{8ThswELzXU3X7Ebd1KdZ7v1rN3GiirRXGKWK099ovBM0FDJCvkopYNQ2a701} \end{align} where   $      \mu_1 = \mu+(\mu_0-\mu-\gamma s)/4$ and $      \mu_2 = \mu+(\mu_0-\mu-\gamma s)/2.$ The $X_\mu$ norm of the trace kernel term  in \eqref{kernel:est} is estimated as \begin{align} &(\mu_0-\mu-\gamma s) \sum_{|\aa|= 2} \norm{  D^{\aa} G(t-s, z,0)B(s)}_{X_\mu} + \sum_{|\aa| \leq 1 } \norm{  D^{\aa} G(t-s, z,0)B(s)}_{X_{\mu_1}} \notag\\ &\quad \quad  \les \frac{1}{\sqrt{t-s}}  \sum_{|\aa_h|\leq1} \left( \lVert\nabla_h^{\aa_h}N(s)\rVert_{\YY_{\mu_1}} + \lVert \nabla_h^{\aa_h}  N(s)\rVert_{S_{\mu_1}} \right)  + \sum_{|\aa_h|\leq 1} \norm{\nabla_h^{\aa_h}  N(s)}_{X_{\mu_1}} \, .    \label{8ThswELzXU3X7Ebd1KdZ7v1rN3GiirRXGKWK099ovBM0FDJCvkopYNQ2a703}    \end{align} Lastly, the initial datum term in \eqref{kernel:est}    may be bounded in the $X_\mu$ norm as \begin{align}    &    \sum_{|\aa|\leq 2} \norm{  D^{\aa} \int_0^\infty G(t,    z,\zz)\omega_{0}(\zz) \,d\zz}_{X_\mu}    \les    \sum_{|\aa|\leq    2}\lVert D^{\aa} \omega_0\rVert_{X_{\mu}}   +    \sum_{|\aa|\leq 2}    \sum_{\xi}      \lVert  \xi^{\aa_h} \partial_{\zz}^{\aa_3}    \omega_{0\xi}\rVert_{L^\infty(z\ge 1+\mu)}     \,.       \llabel{9fd ED ddys 3y1x 52pbiG Lc a 71j G3e uli Ce uzv2 R40Q 50JZUB uK d U3m May 8ThswELzXU3X7Ebd1KdZ7v1rN3GiirRXGKWK099ovBM0FDJCvkopYNQ2a704} \end{align} \colb \end{lemma} \ghoahbkseibhseir The proof of this  lemma uses the following properties of the weight function $w(z)$. \begin{proposition}   \label{wei} The weight function $w$  satisfies \begin{itemize} \item[(a)]  $w(y) \les w(z) $ for $y\le z$, \item[(b)] $w(y) \les w(z)$ for $0 < y/2 \leq z \leq 1+\mu_0$, \item[(c)] $\sqrt{\nu} \les w(y) \les 1$ for $y \in [0,1+\mu_0]$, \item[(d)] $y \les w(y)$ for  $y\in [0,1+\mu_0]$, \item[(e)] $w(y) e^{-\frac{y}{C \sqrt{\nu}}}  \les \sqrt{\nu}$ for $y\in [0,1+\mu_0]$ where $C>0$ is sufficiently large constant, depending only on $\theta_0$ in \eqref{8ThswELzXU3X7Ebd1KdZ7v1rN3GiirRXGKWK099ovBM0FDJCvkopYNQ2a101}. \end{itemize} \end{proposition}   \ghoahbkseibhseir The next lemma deals with the $\YY_\mu$ analytic norm. \begin{lemma}[\bf Main $\YY$ norm estimate] \label{lem:main:Y} \cole Let $\mu_1$ be as defined in Lemma~\ref{lem:main:X}. Then the  nonlinear term in \eqref{kernel:est} is bounded in the $\YY_\mu$ norm as \begin{align} & (\mu_0-\mu-\gamma s) \sum_{|\aa|= 2} \left\lVert D^{\aa} (1+ |\nabla_h|)\int_0^\infty G(t-s, z, \zz)N(s, \zz) \,d\zz\right\rVert_{\YY_\mu} \notag\\ &\quad + \sum_{|\aa|\leq 1}  \left\lVert D^{\aa} (1+ |\nabla_h|) \int_0^\infty G(t-s, z, \zz)N(s, \zz) \,d\zz\right\rVert_{\YY_{\mu_1}} \notag\\ &\quad  \quad  \les \sum_{|\aa|\leq 1}\lVert D^{\aa} (1+ |\nabla_h|) N(s)\rVert_{\YY_{\mu_1}}    +\sum_{|\aa|\leq 1}  \lVert\nabla^{\aa} (1+ |\nabla_h|) N(s)\rVert_{S_{\mu_1}} \, . \label{8ThswELzXU3X7Ebd1KdZ7v1rN3GiirRXGKWK099ovBM0FDJCvkopYNQ2a706} \end{align} The $\YY_\mu$ norm of the trace kernel term in \eqref{kernel:est} is estimated as  \begin{align} & (\mu_0-\mu - \gamma s) \sum_{|\aa| = 2} \norm{ D^{\alpha} (1+ |\nabla_h|) G(t-s, z,0)B(s)}_{\YY_\mu} + \sum_{|\aa|\leq 1} \norm{ D^{\alpha} (1+ |\nabla_h|) G(t-s, z,0)B(s)}_{\YY_{\mu_1}} \notag \\ & \quad \les  \sum_{i\leq1}\left( \lVert\nabla_h^{\aa_h} (1+ |\nabla_h|) N(s)\rVert_{\YY_{\mu_1}} + \lVert \nabla_h^{\aa_h} (1+ |\nabla_h|) N(s)\rVert_{S_\mu}   \right) \, . \llabel{0uo S7 ulWD h7qG 2FKw2T JX z BES 2Jk Q4U Dy 4aJ2 IXs4 RNH41s py T GNh hk0 8ThswELzXU3X7Ebd1KdZ7v1rN3GiirRXGKWK099ovBM0FDJCvkopYNQ2a707}  \end{align} Lastly, the initial datum term in \eqref{kernel:est} may be bounded as   \begin{align}    &  \sum_{|\aa|\leq 2} \norm{ D^{\alpha} (1+ |\nabla_h|) \int_0^\infty G(t, z,\zz)\omega_{0}(\zz) \,d\zz}_{\YY_\mu}    \notag \\&\indeq    \les    \sum_{|\aa|\leq 2}\lVert D^{\aa} (1+ |\nabla_h|)    \omega_0\rVert_{\YY_{\mu}}   +     \sum_{|\aa|\leq 2} \sum_{\xi}    \lVert  \xi^{\aa_h} \partial_{\zz}^{\aa_3} (1+ |\xi|)    \omega_{0\xi}\rVert_{L^1(z\ge 1+\mu)}     \, .    \label{8ThswELzXU3X7Ebd1KdZ7v1rN3GiirRXGKWK099ovBM0FDJCvkopYNQ2a708}    \end{align} \colb \end{lemma} \ghoahbkseibhseir \begin{remark}  One of the major difference between the above two lemmas and the  corresponding ones in~\cite{KukVicWan19}  is that the vorticity  $\omega$ in 3D is a vector rather than a scalar function in 2D  case. Note from~\eqref{kernel:est} that both cases share similar  integration representation formulas and $G_{\xi}$ in~\eqref{gre:fun}  appears in 2D case essentially ($G_{2,\xi}$ is part of the kernel   $G_{1,\xi}$). Therefore, following the proof in~\cite{KukVicWan19} ,  we can prove Lemma~\ref{lem:main:X}--\ref{lem:main:Y} and we thus  omit the details.  \end{remark}   \ghoahbkseibhseir  \ghoahbkseibhseir \section{The Sobolev norm estimate}  \label{sec-sobolev} In this section, we state the estimates for   the  Sobolev part of the norm   \begin{equation}    \sum_{|\aa|\le 5}  \norm{\nabla^{\aa} \omega}_{S}    = \sum_{|\aa|\le 5}  \norm{\nabla^{\aa} \omega}_{L^2_{x,y}(y\geq 1/2)}   = \sum_{|\aa|\le 5}\left(\sum_\xi\lVert z |\xi|^{|\alpha_h|}\partial_{z}^{\aa_3}\omega_\xi\rVert^2_{L^2(y\ge1/2)}\right)^{1/2}    \,    \label{8ThswELzXU3X7Ebd1KdZ7v1rN3GiirRXGKWK099ovBM0FDJCvkopYNQ2a275}    \end{equation} where $\nabla = (\nabla_h, \partial_{z}) =    (\partial_{x}, \partial_{y}, \partial_{z})$. The proofs in the $3$D    case are similar as the 2D case and are omitted.  For a given norm    $\norm{\cdot}$, it is  convenient to introduce the notation    \begin{align} \lVert \nabla^k u\rVert = \sum_{|\aa| = k} \lVert    \nabla^{\aa} u \rVert = \sum_{|\aa| = k} \lVert     \partial_{x}^{\aa_1}\partial_{y}^{\aa_2}\partial_{z}^{\aa_3} u    \rVert \, . \notag \end{align} \ghoahbkseibhseir We first state a lemma which estimates $u$ in terms of $\omega$.  \ghoahbkseibhseir \cole \begin{lemma} \label{x:u} Let $t$ be such that $\gamma t \leq \mu_0/2$. Then we have   \begin{equation}  \sum_{0\leq k \leq 3} \lVert \nabla^k u(t) \rVert_{L^\infty(z\ge1/4)}     \les \sum_{|\aa| \leq 3} \sum_\xi \lVert |\xi|^{|\alpha_h|}\partial_{z}^{\aa_3} u_\xi(t) \rVert_{L^\infty(z\ge1/4)}     \lesssim     \NORM{\omega}_t    \llabel{w5Z C8 B3nU Bp9p 8eLKh8 UO 4 fMq Y6w lcA GM xCHt vlOx MqAJoQ QU 1 e8a 2aX 8ThswELzXU3X7Ebd1KdZ7v1rN3GiirRXGKWK099ovBM0FDJCvkopYNQ2a281}   \end{equation}   and    \begin{align}  \norm{D^5 u(t)}_{L^2(z\geq   1/4)}  \les     \NORM{\omega}_t  \, .  \llabel{9Y6 2r lIS6 dejK   Y3KCUm 25 7 oCl VeE e8p 1z UJSv bmLd Fy7ObQ FN l J6F RdF   8ThswELzXU3X7Ebd1KdZ7v1rN3GiirRXGKWK099ovBM0FDJCvkopYNQ2a281b}   \end{align} \end{lemma} \colb \ghoahbkseibhseir \ghoahbkseibhseir Next we state an a priori estimate for the norm $\sum_{|\aa| \leq 4} \lVert{\nabla^{\aa} \omega \rVert}_S$. We denote  \begin{align} \phi(z)=z\bar \psi(z)    \, , \llabel{kEm qM N0Fd NZJ0 8DYuq2 pL X JNz  4rO ZkZ X2 IjTD 1fVt z4BmFI Pi 0 GKD R2W 8ThswELzXU3X7Ebd1KdZ7v1rN3GiirRXGKWK099ovBM0FDJCvkopYNQ2a:phi:def} \end{align} where $\bar\psi\in C^\infty$  is a non-decreasing function such that $\bar\psi=0$ for $0\le z\le 1/4$ and $\bar\psi=1$ for $z\ge 1/2$. Noting that  \begin{align} \norm{z f}_{L^2(z\geq 1/2)} \leq \norm{\phi f}_{L^2(\HH)} \,, \notag \end{align} it suffices to estimate the norm defined in \eqref{8ThswELzXU3X7Ebd1KdZ7v1rN3GiirRXGKWK099ovBM0FDJCvkopYNQ2a275} with the weight $z$ changed to $\phi$.  \ghoahbkseibhseir \ghoahbkseibhseir \cole \begin{lemma} \label{L12a} For any $0 < t < \frac{\mu_0}{2\gamma}$, the estimate    \begin{align} &\sum_{|\aa|\le 5} \lVert \phi \nabla^{\aa} \omega(t)\rVert^2_{L^2(\HH)}   \notag \lesssim   \Bigl(1+ t \sup_{s\in[0,t]} \NORM{\omega(s)}_s^3 \Bigr) e^{C t (1+ \sup_{s\in[0,t]} \NORM{\omega(s)}_s)}    \sum_{|\aa|\le 5} \lVert \phi \nabla^{\aa} \omega_0\rVert^2_{L^2(\HH)}    \llabel{PhO zH zTLP lbAE OT9XW0 gb T Lb3 XRQ qGG 8o 4TPE 6WRc uMqMXh s6 x Ofv 8st  8ThswELzXU3X7Ebd1KdZ7v1rN3GiirRXGKWK099ovBM0FDJCvkopYNQ2a286a} \end{align} holds,  where $C>0$ is a constant independent of $\gamma$. \end{lemma} \colb \ghoahbkseibhseir As a direct consequence of the above result,   \begin{align}    \sum_{|\aa|\le 5} \lVert z    \nabla^{\aa}\omega(t)\rVert^2_{L^2(z\geq 1/2)}          \lesssim    \Bigl(1+ t \sup_{s\in[0,t]} \NORM{\omega(s)}_s^3 \Bigr)      e^{C t    (1+ \sup_{s\in[0,t]} \NORM{\omega(s)}_s)}    \sum_{|\aa|\le 5}    \lVert z  \nabla^{\aa}\omega_0\rVert^2_{L^2(z\geq 1/4)}    \label{8ThswELzXU3X7Ebd1KdZ7v1rN3GiirRXGKWK099ovBM0FDJCvkopYNQ2a286}    \end{align}  holds for the same constant in    Lemma~\ref{L12a}. \ghoahbkseibhseir \section{Inviscid limit}    \label{con} \subsection{Closing the a~priori estimates} \label{sec05} We may close the a priori estimate in a similar manner as in \cite{KukVicWan19}. However, for completeness, we provide a sketch of the proofs of Theorems~\ref{T01} and  \ref{T02}. \ghoahbkseibhseir \begin{proof}[Proof of Theorem~\ref{T01}] We denote \begin{align} \tilde M  = \sum_{|\aa|\leq 2}\lVert D^{\aa} \omega_0\rVert_{X_{\mu_0}} +      \sum_{|\aa|\leq 2} \sum_{\xi} \lVert \nabla^{\aa}\omega_{0,\xi}\rVert_{L^\infty(z\ge 1+\mu_0)} \notag \end{align} and  \begin{align} \overline M &=  \sum_{|\aa|\leq 2}\lVert D^{\aa} (1+ |\nabla_h|)\omega_0\rVert_{\YY_{\mu_0}}   +     \sum_{|\aa|\leq 2} \sum_{\xi}            \lVert  \nabla^{\aa}(1+ |\nabla_h|)\omega_{0,\xi}\rVert_{L^1(z\ge 1+\mu_0)}             \, . \notag \end{align}   \ghoahbkseibhseir First we estimate the $X(t)$ norm of $\omega(t)$.  From the mild formulation \eqref{kernel:est}, the estimates \eqref{8ThswELzXU3X7Ebd1KdZ7v1rN3GiirRXGKWK099ovBM0FDJCvkopYNQ2a701}--\eqref{8ThswELzXU3X7Ebd1KdZ7v1rN3GiirRXGKWK099ovBM0FDJCvkopYNQ2a703} for the nonlinear term, and the Lemma \ref{L01}, we obtain   \begin{align}  \sum_{|\aa|=2} \norm{ D^{\aa} \omega(t)}_{X_\mu}  &\les   \int_0^t  \left(\frac{\NORM{\omega(s)}_s^2}{(\mu_0-\mu-\gamma s)^{3/2+\AA}}  +   \frac{1}{\sqrt{t-s}} \frac{\NORM{\omega(s)}_s^2}{(\mu_0-\mu-\gamma  s)^{1+\AA}}\right) \,ds +  \tilde M   \notag\\ &    \les  \sup_{0\leq  s \leq t} \NORM{\omega(s)}_s^2  \left( \frac{1}{\gamma  (\mu_0-\mu-\gamma t)^{1/2+\AA}}  +  \frac{1}{\sqrt{\gamma}  (\mu_0-\mu-\gamma t)^{1/2+\AA}}   \right)  +  \tilde M    \notag\\  &\les        \frac{\sup_{0\leq s \leq t}  \NORM{\omega(s)}_s^2}{\sqrt{\gamma} (\mu_0-\mu-\gamma t)^{1/2+\AA}}  + \tilde M    \,   \label{8ThswELzXU3X7Ebd1KdZ7v1rN3GiirRXGKWK099ovBM0FDJCvkopYNQ2a305}   \end{align} where $t < \mu_0/2\gamma$, $s\in (0,t)$, and $\mu <   \mu_0-\gamma t$.  Similarly, we obtain \begin{align}   \sum_{|\aa|\leq 1} \norm{ D^{\aa} \omega(t)}_{X_\mu}  &\les   \int_0^t \biggl(\frac{\NORM{\omega(s)}_s^2}{(\mu_0-\mu-\gamma   s)^{1/2+\AA}}  + \frac{1}{\sqrt{t-s}}   \frac{\NORM{\omega(s)}_s^2}{(\mu_0-\mu-\gamma s)^{\AA}} \biggr) \,ds   +  \tilde M   \notag\\       &\les      \frac{\sup_{0\leq s \leq t} \NORM{\omega(s)}_s^2}{\sqrt{\gamma}}      + \tilde M    \, .    \label{8ThswELzXU3X7Ebd1KdZ7v1rN3GiirRXGKWK099ovBM0FDJCvkopYNQ2a305a}    \end{align}  Combining    \eqref{8ThswELzXU3X7Ebd1KdZ7v1rN3GiirRXGKWK099ovBM0FDJCvkopYNQ2a305}    and    \eqref{8ThswELzXU3X7Ebd1KdZ7v1rN3GiirRXGKWK099ovBM0FDJCvkopYNQ2a305a},    we obtain \begin{align} \norm{\omega(t)}_{X(t)} \les    \frac{\sup_{0\leq s \leq t} \NORM{\omega(s)}_s^2}{\sqrt{\gamma}}    + \tilde M    \, .    \label{8ThswELzXU3X7Ebd1KdZ7v1rN3GiirRXGKWK099ovBM0FDJCvkopYNQ2a305b}    \end{align} \ghoahbkseibhseir Next we estimate the $\YY(t)$ norm of $\omega(t)$.  From the mild formulation \eqref{kernel:est}, the estimates \eqref{8ThswELzXU3X7Ebd1KdZ7v1rN3GiirRXGKWK099ovBM0FDJCvkopYNQ2a706}--\eqref{8ThswELzXU3X7Ebd1KdZ7v1rN3GiirRXGKWK099ovBM0FDJCvkopYNQ2a708} for the nonlinear term, and Lemma~\ref{L09}, , we obtain \begin{align}  \sum_{|\aa|=2} \norm{ D^{\aa} (1+  |\nabla_h|)\omega(t)}_{\YY_\mu}  &\les   \int_0^t \frac{\NORM{\omega(s)}_s^2}{(\mu_0-\mu-\gamma s)^{1+\AA}}  \,ds + \overline M    \les      \frac{\sup_{0\leq s \leq t} \NORM{\omega(s)}_s^2}{ \gamma (\mu_0-\mu-\gamma t)^{\AA}}   + \overline M    \, . \label{8ThswELzXU3X7Ebd1KdZ7v1rN3GiirRXGKWK099ovBM0FDJCvkopYNQ2a306}    \end{align} For the lower order derivatives we obtain \begin{align} \sum_{|\aa|\leq 1} \norm{ D^{\aa}(1+ |\nabla_h|) \omega(t)}_{\YY_\mu} &\les   \int_0^t \frac{\NORM{\omega(s)}_s^2}{(\mu_0-\mu-\gamma s)^{\AA}}  \,ds +  \overline M      \les      \frac{\sup_{0\leq s \leq t} \NORM{\omega(s)}_s^2}{ \gamma}   + \overline M    \, . \label{8ThswELzXU3X7Ebd1KdZ7v1rN3GiirRXGKWK099ovBM0FDJCvkopYNQ2a306a} \end{align} By combining \eqref{8ThswELzXU3X7Ebd1KdZ7v1rN3GiirRXGKWK099ovBM0FDJCvkopYNQ2a306}--\eqref{8ThswELzXU3X7Ebd1KdZ7v1rN3GiirRXGKWK099ovBM0FDJCvkopYNQ2a306a}, we arrive at \begin{align}   \norm{\omega(t)}_{\YY(t)} \les  \frac{\sup_{0\leq s \leq t} \NORM{\omega(s)}_s^2}{\gamma}     + \overline M    \, . \label{8ThswELzXU3X7Ebd1KdZ7v1rN3GiirRXGKWK099ovBM0FDJCvkopYNQ2a306b}  \end{align} To conclude, let \begin{align*} \mathring M = \sum_{|\aa|\le 3} \lVert  \nabla^{\aa}\omega_0\rVert_{L^2(z\ge 1/4)} \les M  \, . \end{align*} Recall that the Sobolev estimate  \eqref{8ThswELzXU3X7Ebd1KdZ7v1rN3GiirRXGKWK099ovBM0FDJCvkopYNQ2a286} yields \begin{align} \norm{\omega(t)}_{\ZZZ}  \les \left(1+ \frac{\sup_{s\in[0,t]} \NORM{\omega(s)}_s^{3/2}}{\sqrt{\gamma}} \right)      e^{\frac{C}{\gamma}  (1+ \sup_{s\in[0,t]} \NORM{\omega(s)}_s)}  \mathring M    \, , \label{8ThswELzXU3X7Ebd1KdZ7v1rN3GiirRXGKWK099ovBM0FDJCvkopYNQ2a307} \end{align} and this inequality holds pointwise in time for $t < \mu_0/2\gamma$. The constant $C$ and the implicit constants in $\les$ are independent of $\gamma$.  \ghoahbkseibhseir Combining \eqref{8ThswELzXU3X7Ebd1KdZ7v1rN3GiirRXGKWK099ovBM0FDJCvkopYNQ2a305b}, \eqref{8ThswELzXU3X7Ebd1KdZ7v1rN3GiirRXGKWK099ovBM0FDJCvkopYNQ2a306b}, and \eqref{8ThswELzXU3X7Ebd1KdZ7v1rN3GiirRXGKWK099ovBM0FDJCvkopYNQ2a307},  and taking the supremum in time for $t < \frac{\mu_0}{2\gamma}$, we arrive at \begin{align} \sup_{t \in [0,\frac{\mu_0}{2\gamma}]} \NORM{\omega(t)}_t &\leq C (\tilde M + \overline M ) + \frac{C \sup_{t \in [0,\frac{\mu_0}{2\gamma}]} \NORM{\omega(t)}_t^2}{\sqrt{\gamma}} \notag\\ &\quad + C \mathring M \left(1+ \frac{\sup_{t \in [0,\frac{\mu_0}{2\gamma}]} \NORM{\omega(t)}_t^{3/2}}{\sqrt{\gamma}} \right)      e^{\frac{C \mu_0}{\gamma}  (1+ \sup_{t \in [0,\frac{\mu_0}{2\gamma}]} \NORM{\omega(t)}_t)}    \, , \notag \end{align} where $C\geq 1$ is a constant that depends only on $\mu_0$.  Using a standard barrier argument, one may show that if $\gamma$ is chosen sufficiently large, in terms of $\tilde M, \overline M, \mathring M, \mu_0$,  we obtain \begin{align} \sup_{t \in [0,\frac{\mu_0}{2\gamma}]} \NORM{\omega(t)}_t \leq 2 C( \tilde M + \overline M  + \mathring M ) \, , \notag \end{align} concluding the proof of the theorem. \end{proof} \ghoahbkseibhseir  \begin{remark} \label{R08} In order to justify the above a~priori estimates, for each $\delta \in (0,1]$, we apply them on the approximate system   \begin{align}   \omega_t + u^{\delta}\cdot\nabla\omega -\nu\Delta\omega   =\omega\cdot\nabla u^{\delta}    \,  ,   \label{8ThswELzXU3X7Ebd1KdZ7v1rN3GiirRXGKWK099ovBM0FDJCvkopYNQ2a03}   \end{align} where $u^{\delta}$ is a regularization of    the   velocity in the Biot-Savart law   \eqref{8ThswELzXU3X7Ebd1KdZ7v1rN3GiirRXGKWK099ovBM0FDJCvkopYNQ2a123}--\eqref{8ThswELzXU3X7Ebd1KdZ7v1rN3GiirRXGKWK099ovBM0FDJCvkopYNQ2a124}. The   boundary conditions \eqref{bdry:hor} and \eqref{bdry:ver} become   \begin{equation} \notag   \nu(\partial_{3}+\Lambda_h)\omega_h|_{z=0}   =\partial_{3}(-\Delta)^{-1}(-u^{\delta}\cdot\nabla\omega_h+\omega\cdot\nabla   u^{\delta}_h)|_{z=0}   \end{equation}  with $  \omega_3|_{z=0}=0$, and the initial condition  is  replaced  by an analytic approximation. Then we apply the a prior estimates on  the regularized solution to the modified   system~\eqref{8ThswELzXU3X7Ebd1KdZ7v1rN3GiirRXGKWK099ovBM0FDJCvkopYNQ2a03}  and pass those bounds to the limit $\delta\to 0$. \end{remark}  \ghoahbkseibhseir Deducing Theorem~\ref{T02} from Theorem~\ref{T01}  is the same as in the 2D case~\cite{KukVicWan19}, and we thus omit  the details. \ghoahbkseibhseir \ghoahbkseibhseir \ghoahbkseibhseir  

\end{document}